\documentclass[11pt]{amsart}
\usepackage{amsmath,amsfonts,mathrsfs}
\usepackage{amssymb}

\usepackage[margin=1in]{geometry}
\usepackage{url}
\usepackage[shortlabels]{enumitem}

\usepackage[unicode,bookmarks]{hyperref}
\usepackage[usenames,dvipsnames]{xcolor}
\hypersetup{colorlinks=true,citecolor=NavyBlue,linkcolor=Brown,urlcolor=Orange}

\usepackage{tikz-cd,adjustbox}
\usepackage[arrow,curve,matrix]{xy}

\newif\ifdraft
\draftfalse

\newcommand{\tensor}{\otimes}
\newcommand{\isom}{\simeq}

\newcommand{\C}{\mathbb{C}}

\newcommand{\Q}{\mathbb{Q}}
\newcommand{\Z}{\mathbb{Z}}

\renewcommand{\P}{\mathbb{P}}

\newcommand{\M}{\mathcal{M}}

\renewcommand{\H}{\mathcal{H}}
\newcommand{\HH}{\mathbb{H}}

\renewcommand{\O}{\mathcal{O}}
\newcommand{\DR}{\mathrm{DR}}

\newcommand{\codim}{\mathrm{codim}}

\newcommand{\K}{\mathcal{K}}
\newcommand{\gr}{\mathrm{gr}}

\newcommand{\h}{\underline h}
\newcommand{\cl}{\mathrm{cl}}

\newcommand{\DB}{\underline{\Omega}} %for du Bois complex
\newcommand{\Sing}{\mathrm{Sing}}
\newcommand{\dual}{\mathbf{D}}

\newcommand{\lcdef}{\mathrm{lcdef}}
\newcommand{\sm}{\smallsetminus}
\newcommand{\Div}{\mathrm{Div}}
\newcommand{\CDiv}{\mathrm{CDiv}}
\newcommand{\Pic}{\mathrm{Pic}}
\newcommand{\IC}{\mathrm{IC}}

\counterwithin{equation}{subsection}
\counterwithout{subsection}{section}
\counterwithin{figure}{subsection}

\newtheorem{thm}[equation]{Theorem}
\newtheorem{cor}[equation]{Corollary}
\newtheorem{lem}[equation]{Lemma}
\newtheorem{prop}[equation]{Proposition}
\newtheorem{conj}[equation]{Conjecture}

\theoremstyle{definition}

\newtheorem{question}[equation]{Question}

\theoremstyle{remark}
\newtheorem{rmk}[equation]{Remark}
\newtheorem{ex}[equation]{Example}

\theoremstyle{plain}

\newcommand{\theoremref}[1]{\hyperref[#1]{Theorem~\ref*{#1}}}
\newcommand{\lemmaref}[1]{\hyperref[#1]{Lemma~\ref*{#1}}}
\newcommand{\definitionref}[1]{\hyperref[#1]{Definition~\ref*{#1}}}
\newcommand{\propositionref}[1]{\hyperref[#1]{Proposition~\ref*{#1}}}
\newcommand{\conjectureref}[1]{\hyperref[#1]{Conjecture~\ref*{#1}}}
\newcommand{\corollaryref}[1]{\hyperref[#1]{Corollary~\ref*{#1}}}
\newcommand{\exampleref}[1]{\hyperref[#1]{Example~\ref*{#1}}}

\makeatletter
\let\old@caption\caption
\renewcommand*{\caption}[1]{%
	\setcounter{figure}{\value{equation}}%
	\stepcounter{equation}%
	\old@caption{#1}\relax%
}
\makeatother

\newcounter{intro}

\newtheorem{intro-conjecture}[intro]{Conjecture}
\newtheorem{intro-corollary}[intro]{Corollary}
\newtheorem{intro-theorem}[intro]{Theorem}

\makeatletter
\def\Ddots{\mathinner{\mkern1mu\raise\p@
\vbox{\kern7\p@\hbox{.}}\mkern2mu
\raise4\p@\hbox{.}\mkern2mu\raise7\p@\hbox{.}\mkern1mu}}
\makeatother

\begin{document}

\title[Q-factoriality and Hodge-Du Bois theory]{Q-factoriality and Hodge-Du Bois theory}

\author{Sung Gi Park}
\address{Department of Mathematics, Princeton University, Fine Hall, Washington Road, Princeton, NJ 08544, USA}
\address{Institute for Advanced Study, 1 Einstein Drive, Princeton, NJ 08540, USA}
\email{sp6631@princeton.edu \,\,\,\,\,\,sgpark@ias.edu}

\author{Mihnea Popa}
\address{Department of Mathematics, Harvard University, 1 Oxford Street, Cambridge, MA 02138, USA}
\email{mpopa@math.harvard.edu}

\thanks{MP was partially supported by NSF grants DMS-2040378 and DMS-2401498.}


\date{\today}

\begin{abstract}
We prove Hodge-theoretic formulas for the $\Q$-factoriality defect of a normal projective variety, and for the local analytic $\Q$-factoriality defect 
of an analytic germ of a normal variety. These formulas lead to consequences ranging from a local analytic version of Samuel's conjecture to a characterization of projective rational homology threefolds with rational singularities, or the invariance of Hodge-Du Bois numbers under flops of projective threefolds.
\end{abstract}

\maketitle

\makeatletter
\newcommand\@dotsep{4.5}
\def\@tocline#1#2#3#4#5#6#7{\relax
  \ifnum #1>\c@tocdepth % then omit
  \else
    \par \addpenalty\@secpenalty\addvspace{#2}%
    \begingroup \hyphenpenalty\@M
    \@ifempty{#4}{%
      \@tempdima\csname r@tocindent\number#1\endcsname\relax
    }{%
      \@tempdima#4\relax
    }%
    \parindent\z@ \leftskip#3\relax
    \advance\leftskip\@tempdima\relax
    \rightskip\@pnumwidth plus1em \parfillskip-\@pnumwidth
    #5\leavevmode\hskip-\@tempdima #6\relax
    \leaders\hbox{$\m@th
      \mkern \@dotsep mu\hbox{.}\mkern \@dotsep mu$}\hfill
    \hbox to\@pnumwidth{\@tocpagenum{#7}}\par
    \nobreak
    \endgroup
  \fi}
\def\l@section{\@tocline{1}{0pt}{1pc}{}{\bfseries}}
\def\l@subsection{\@tocline{2}{0pt}{25pt}{5pc}{}}
\makeatother

%========================================================

\tableofcontents

\section{Introduction}\label{scn:intro}

In this paper we relate $\Q$-factoriality, and its local analytic version in the case of threefolds, to various symmetries of the singular cohomology or the Hodge-Du Bois diamond of a complex projective variety. More generally, on normal varieties we obtain Hodge-theoretic formulas for the defect of $\Q$-factoriality of a projective variety, and for the defect of local analytic $\Q$-factoriality of an analytic germ.
This complements the picture described in the companion paper \cite{PP25a}, where symmetries in the cohomology of a projective variety are related to its local cohomological defect and higher rational singularities.\footnote{Note that this paper and \cite{PP25a} appeared originally as a single paper \cite{PP}. Both contain new material, compared to the original preprint.} At the same time, it is a line of study of independent interest that has been considered previously in more specialized cases.

\noindent
{\bf $\Q$-factoriality defect.}
We prove a general result showing that a potential symmetry of the singular cohomology of a normal projective variety $X$ is controlled by its $\Q$-factoriality. This has appeared before in the literature in various special situations, most importantly in the case of threefolds, in the work of Namikawa and Steenbrink \cite{NS95}.

Following the usual conventions, for a normal variety $X$, we denote by $\Div(X)$ the free abelian group of Weil divisors on $X$ (i.e. generated by the irreducible codimension $1$ subvarieties of $X$), and by $\CDiv(X)$ the subgroup of Cartier divisors on $X$. We also denote $\Div_\Q(X):=\Div(X)\otimes_\Z\Q$ and $\CDiv_\Q(X):=\CDiv(X)\otimes_\Z\Q$. The subgroups $\mathrm{Alg}^1_\Q(X)$, $\mathrm{Hom}^1_\Q(X)$, and $\mathrm{Num}^1_\Q(X)$ of $\Div_\Q(X)$ are the groups of $\Q$-Weil divisors algebraically, homologically, or numerically equivalent to zero, respectively; see \cite[Section 19.1]{Fulton98}. The \textit{$\Q$-factoriality defect} of $X$ is 
$$
\sigma(X):=\dim_\Q \Div_\Q(X)/\CDiv_\Q(X).
$$
Thus $\sigma(X)=0$ if and only if $X$ is $\Q$-factorial. This invariant was originally defined and studied by Kawamata \cite{Kawamata88} in order to prove important results in the three-dimensional Minimal Model Program, and has since been a subject of interest.  When $X$ is a projective threefold with at worst isolated rational singularities and $H^2 (X, \O_X)=0$, Namikawa-Steenbrink \cite[Theorem 3.2]{NS95} proved that $\sigma(X)= h^4 (X)- h^2(X)$,\footnote{We use the notation $h^i (X) : = \dim_{\Q} H^i (X, \Q)$ rather than the more common Betti number notation $b_i (X)$ for uniformity, since at times we will make comparisons with the dimension $Ih^i (X) : = \dim_{\Q} \mathit{IH}^i (X, \Q)$ of intersection cohomology.} thus providing an obstruction to Poincar\'e duality, with numerous applications to the problem of the existence of global smoothings, and to the study of terminal (especially Fano) threefolds. See also \cite[Remark 3.2]{FL22} for a modern interpretation of this formula using link invariants.

Here we provide an extension of this result to arbitrary normal varieties of any dimension. This will be used in many of the applications discussed later.
 
 \begin{intro-theorem}\label{thm:Q-factoriality-defect}
Let $X$ be a normal projective variety of dimension $n$. Then the following conditions are equivalent:
\begin{center}
(i) $h^1(X)=h^{2n-1}(X)$,\quad
(ii) $\mathrm{Alg}^1_\Q(X)\subset \CDiv_\Q(X)$,\quad
(iii) $\sigma(X)$ is finite.
\end{center}
In this case,
$$
\sigma(X)=\dim_\Q \big(\mathit{IH}^2(X,\Q)\cap \mathit{IH}^{1,1}(X)\big)-\dim_\Q \big(\ker(H^2(X,\Q)\to H^2(X,\O_X))\big).
$$
In particular, if $X$ has rational singularities, then
$$
\sigma(X)= h^{2n-2}(X)-h^2 (X).
$$
\end{intro-theorem}

Thus for a variety with rational singularities, a small piece of all potential duality, namely that between $H^2 (X, \Q)$ and $H^{2n-2} (X, \Q)$, implies that $X$ must be $\Q$-factorial. It is worth mentioning that condition (ii) implies that every deformation of a $\Q$-Cartier divisor is $\Q$-Cartier.

The proof of Theorem \ref{thm:Q-factoriality-defect} relies on a version of the Lefschetz $(1,1)$-theorem for Weil divisors on normal varieties, which states that there is
a surjective cycle class morphism 
$$
{\rm cl}_\Q \colon \Div_\Q(X)\to \mathit{IH}^2(X,\Q)\cap \mathit{IH}^{1,1}(X)
$$
whose kernel is $\mathrm{Alg}^1_\Q(X) = \mathrm{Hom}^1_\Q(X)$. This is essentially already known, although the description of the target space is new. As a supplement, we show that if $X$ is $\Q$-factorial, algebraic and homological equivalence for $\Q$-Weil divisors also coincide with numerical equivalence, i.e. the kernel is also 
equal to $\mathrm{Num}^1_\Q(X)$. See Theorem \ref{thm:Lefschetz-(1,1)} and the surrounding discussion.

\smallskip
\noindent
{\bf Local analytic $\Q$-factoriality defect.}
It is important to also consider an analogous but more refined local concept, that has interesting connections with global properties. Following the notation in \cite[Section 1]{Kawamata88}, for a normal analytic variety $X$ of dimension $n$, and for a point $x \in X$, one can define 
$$
\sigma^{\rm an}(X;x) :=\dim_\Q \Div_\Q^{\rm an}(X,x)/\CDiv_\Q^{\rm an}(X,x),
$$
the \emph{local analytic $\Q$-factoriality defect} of $X$ at $x$. Here $\Div_\Q^{\rm an}(X,x)$ is the direct limit of $\Div_\Q^{\rm an}(U)$, over all 
the analytic open neighborhoods of $x$ under inclusion, and similarly for $\CDiv_\Q^{\rm an}(X,x)$. This is equal to $0$ if and only if $X$ is locally analytically $\Q$-factorial at $x$ (which in turn implies $\Q$-factoriality, if true at all points). Theorem  \ref{thm:Q-factoriality-defect} has a local analogue:

\begin{intro-theorem}\label{thm:local analytic Q-factoriality defect}
Let $x\in X$ be an analytic germ of a normal algebraic variety of dimension $n$. Then the following conditions are equivalent:
\begin{center}
(i) $\sigma^{\rm an}(X;x)$ is finite,\quad
(ii) $R^1\mu_*\O_{\widetilde X}=0$ for a resolution $\mu:\widetilde X\to X$.
\end{center}
In this case,
$$
\sigma^{\rm an}(X;x)=\dim_\Q \ker\left(\H^{-n+2}(\IC_X)_x\to (R^2\mu_*\O_{\widetilde X})_x\right).
$$
In particular, if $X$ has rational singularities, then
$$
\sigma^{\rm an}(X;x)=\dim_\Q  \H^{-n+2}(\IC_X)_x.
$$
\end{intro-theorem}

The last statement also follows from \cite[Proposition 12.1]{PP25a}, which makes essential use of \cite[Satz 6.1]{Flenner81}.
Here $\IC_X$ is the intersection complex of $X$ with rational coefficients, and $\H^\bullet$ denotes its constructible cohomology. The theorem has some basic consequences described in Section \ref{scn:loc-an-def}; for instance, in Corollary \ref{cor:infinite-defect} we show that if $X$ is $S_3$, the set of points where $\sigma^{\rm an} (X; x)$ is infinite is either empty or a Zariski closed subset of pure codimension $2$.

The proof of Theorem \ref{thm:local analytic Q-factoriality defect} also relies on a version, this time local analytic, and perhaps more surprising, 
of the Lefschetz $(1,1)$-theorem; this is Theorem \ref{thm:local Lefschetz-(1,1)} in the main body of the paper. A key role in the statement and approach is played by the intersection complex, and the proof relies on the theory of mixed Hodge modules.

\noindent
{\bf Analogues of Samuel's conjecture.}
Recall that Samuel's conjecture states that a local complete intersection $X$ which is factorial in codimension $3$ is in fact factorial; this was proved by Grothendieck \cite[Corollaire 3.14]{SGA2}. In particular, this automatically applies if $X$  is regular in codimension $3$.

Theorem \ref{thm:local analytic Q-factoriality defect} quickly leads to various versions of this statement involving (local analytic) $\Q$-factoriality. The first is the local analytic version of Samuel's conjecture; a stronger version is stated in Theorem \ref{thm:Samuel conjecture}.

\begin{intro-corollary}
\label{cor:Samuel conjecture}
If $X$ is an algebraic local complete intersection, then $X$ is locally analytically $\Q$-factorial if and only if it is locally analytically $\Q$-factorial away from a subvariety of codimension at least $4$.
\end{intro-corollary}

Recall that if $X$ is a local complete intersection, then the local cohomological defect of $X$ (see Section \ref{scn:background})  satisfies $\lcdef (X) = 0$; the latter is a more general condition. Another consequence of Theorem \ref{thm:local analytic Q-factoriality defect} is the following statement for a variety with Du Bois singularities or Serre's condition $S_3$; a generalization to arbitrary $\lcdef (X)$ is stated in Corollary \ref{cor:Samuel-lcdef}.

\begin{intro-corollary}\label{cor:Samuel-lcdef-0}
Let $X$ be a normal algebraic variety with $\lcdef (X) = 0$ and $\codim_X \,\Sing(X) \ge 4$. 
If $X$ is Du Bois or satisfies $S_3$,  then $X$ is locally analytically $\Q$-factorial.
\end{intro-corollary}

\noindent
{\bf Local analytic $\Q$-factoriality defect and Hodge-Du Bois numbers.}
In the case of projective threefolds with rational singularities, using the two main theorems above as well as results from \cite{PP25a}, we give a formula for the local analytic $\Q$-factoriality defect in terms of Hodge-Du Bois numbers, completing the picture provided by Theorem \ref{thm:Q-factoriality-defect}.

\begin{intro-theorem}\label{thm:analytic-Q-factoriality-threefolds}
Let $X$ be a projective threefold with rational singularities. Then, there are two inequalities of Hodge-Du Bois numbers
$$
\underline h^{2,2}(X)\ge\underline h^{1,1}(X),\quad \underline h^{1,2}(X)\ge\underline h^{2,1}(X),
$$
and the equality
$$
\sigma(X)=\underline h^{2,2}(X)-\underline h^{1,1}(X).
$$
Additionally, $X$ is locally analytically $\Q$-factorial away from a finite set $\left\{x_s\right\}_{s\in S}$ of closed points on $X$, and
$$
\sum_{s\in S}\sigma^{\rm an}(X;x_s)-\sigma(X)=\underline h^{1,2}(X)-\underline h^{2,1}(X).
$$
\end{intro-theorem}

This theorem extends (and rephrases) another result of Namikawa-Steenbrink, \cite[Proposition 3.10]{NS95}, to arbitrary threefolds with rational singularities. The Hodge-Du Bois numbers appearing here are singular analogues of the Hodge numbers of a smooth projective variety, given by 
$$\h^{p,q} (X) := \dim_{\C} \HH^q (X, \DB_X^p),$$
where $\DB_X^p$ is the $p$-th Du Bois complex of $X$; see \S1. In the form 
$$\sigma (X) = \h^{n-1, n-1} (X) - \h^{1,1} (X),$$ 
the formula for $\sigma (X)$ holds in fact in arbitrary dimension; see Remark \ref{rmk:formula-sigma-DB}.

By \cite[Theorem A]{PP25a}, all the other Hodge-Du Bois numbers of a threefold with rational singularities  besides those in Theorem \ref{thm:analytic-Q-factoriality-threefolds} (i.e. those on the boundary of the Hodge diamond) satisfy the usual symmetries from the smooth case. As a consequence, we obtain another proof of \cite[Theorem D]{PP25a} in the projective case.

\begin{intro-corollary}\label{cor:RHM-threefolds}
Let $X$ be a projective threefold with rational singularities. Then the following are equivalent:

\noindent
(i) $X$ is a rational homology manifold. 

\noindent 
(ii) The Hodge-Du Bois diamond of $X$ satisfies full symmetry, i.e.
$$
\underline h^{p,q}(X)=\underline h^{q,p}(X)=\underline h^{3-p,3-q}(X)=\underline h^{3-q,3-p}(X)
$$
for all $0\le p,q\le 3$.

\noindent
(iii) $X$ is locally analytically $\Q$-factorial.
\end{intro-corollary}

Note that conditions (i) and (iii) are local, and their equivalence can be shown to hold even in the non-projective case. See \emph{loc. cit.} for more context on this statement, and for previous results in this direction in  \cite{Kollar89} and  \cite{GW2018}.

\smallskip
\noindent
{\bf Hodge-Du Bois numbers under threefold flips and flops.}
Another application is to determine the behavior of the Hodge-Du Bois numbers of a projective threefold with terminal singularities under flips and flops. By Kontsevich's work, it is known in arbitrary dimension that smooth $K$-equivalent varieties have the same Hodge numbers. This was explained by Batyrev \cite[Theorem 3.4]{Batyrev98}, who extended the statement to the stringy $E$-function of varieties with klt singularities. Note however that this is quite different from considering the Hodge-Du Bois numbers, and in dimension at least $4$ we do not expect all of these to be invariant under flops; cf. Conjecture \ref{conj:higher-dim-flops}.

For threefold terminal flops and flips, Koll\'ar \cite[Corollary 4.12]{Kollar89} proved that intersection cohomology is invariant. On the other hand, results in \cite{PP25a} imply that a projective threefold $X$ with rational singularities satisfies
$$
\underline h^{p,q}(X)=I\underline h^{p,q}(X) \,\,\,\,\,\, {\rm for}\,\,\,\, (p,q)\neq (1,1)\mathrm{\;and\;}(1,2),
$$
where on the right hand side we have the intersection cohomology Hodge numbers. Using Theorem \ref{thm:analytic-Q-factoriality-threefolds} as the new ingredient, we complete the picture for the Hodge-Du Bois numbers:

\begin{intro-theorem}\label{thm:invariance-Hodge}
Let $X$ be a projective threefold with $\Q$-factorial terminal singularities.

\noindent 
(i) Let $g:X\to Z$ be a flopping contraction. For a $\Q$-divisor $D\in \Div_\Q(X)$, let $X^+$ be the $D$-flop of $g$. If $X^+$ is $\Q$-factorial, then we have:
$$
\underline h^{p,q}(X)=\underline h^{p,q}(X^+) \,\,\,\,\,\,{\rm for ~all} \,\,\,\, 0\le p,q\le3.
$$
In particular, the Hodge-Du Bois numbers of two birational minimal models are the same.

\noindent
(ii) Let $g_R: X\to Z$ be the flipping contraction of a $K_X$-negative extremal ray. Let $g_R^+:X^+\to Z$ be the flip. Then,
$$
\underline h^{p,q}(X)=\underline h^{p,q}(X^+)  \,\,\,\,\,\,{\rm for ~all} \,\,\,\, (p,q)\neq (1,2),
$$
and $\underline h^{1,2}(X)\le\underline h^{1,2}(X^+)$.
\end{intro-theorem}

Note that in (i) the $\Q$-factoriality of $X^+$ is immediate if $D$ is effective and $g$ is the contraction of a $(K_X+\epsilon D)$-negative extremal ray (see \cite[Proposition 3.37]{KM98}). The last statement in (i) follows from the fact that a birational map between threefold minimal models can be written as a composition of $\Q$-factorial terminal 
flops, by \cite[Theorem 5.3]{Kawamata88} and \cite[Theorem 4.9]{Kollar89}.

\noindent
{\bf Further applications.}
Theorem \ref{thm:Q-factoriality-defect} has several other interesting consequences related to properties of the $\Q$-factoriality index $\sigma (X)$.
Here is an example, related to its behavior in families;  it generalizes \cite[Proposition 12.1.7]{KM92}, which is a similar statement about $\Q$-factoriality.

\begin{intro-corollary}
In a connected flat projective family of varieties with rational singularities, the set of varieties with fixed $\Q$-factoriality defect is constructible.
\end{intro-corollary}

We leave the other statements of applications for the main text, and only briefly describe them here. 

Still on the topic of deformations, relying on a result from \cite{ST23}, we deduce that klt singularities of pairs deform in a fixed ambient space; see Corollary \ref{cor: klt deforms X fixed}.

In a different direction, Theorem \ref{thm:Q-factoriality-defect} leads to the inequality $\sigma (D) \le \sigma (X)$ for a general hyperplane section $D$ of a variety $X$ of dimension at least $4$; this is already a consequence of \cite{RS06}. It also leads to a new criterion for having the equality $\sigma (D) = \sigma (X)$, which holds in a variety of situations; see Theorem \ref{thm: Bertini for defect of Q-factoriality} and Example \ref{ex:sigma-eq}. 

Another application, this time to the factoriality of hypersurfaces in fourfolds, is given in Section \ref{scn:fourfolds}. This is inspired by the paper 
\cite{PRS14}; we extend its main results from hypersurfaces in $\P^4$ to hypersurfaces in arbitrary smooth projective fourfolds, using Theorem \ref{thm:Q-factoriality-defect} in combination with methods from \emph{loc. cit.}

\smallskip
\noindent
{\bf Acknowledgements.} 
We thank  J\'anos Koll\'ar,  Mircea Musta\c t\u a, Vasudevan Srinivas and Claire Voisin for valuable discussions. We would especially like to thank J\'anos Koll\'ar for providing us an unpublished draft on local Picard groups, from which we extracted Lemma \ref{lem:Kollar}.

\section{\texorpdfstring{$\Q$}{Q}-factoriality and the Lefschetz (1,1)-theorem}

\subsection{Background}\label{scn:background}
This paper was written in conjunction with \cite{PP25a}, which covers the preliminaries needed here in detail. Rather than essentially repeating the introductory sections, we will mostly refer to \cite{PP25a} for details. We only review here some of the main definitions.

Let $X$ be a complex projective variety of dimension $n$. The singular cohomology groups $H^k (X,\Q)$ have a mixed Hodge structure with (increasing) Hodge filtration $F_\bullet$. The \emph{Hodge-Du Bois numbers} $\underline h^{p,q}$ are defined by
$$
\underline h^{p,q}(X):=\dim_\C \gr^F_{-p}H^{p+q}(X,\C).
$$
By the degeneration of the generalized Hodge-de  Rham  spectral sequence, we have isomorphisms 
$$H^{p,q} (X): = \gr^F_{-p}H^{p+q}(X,\C) \simeq  \HH^q(X,\DB_X^p),$$
where $\DB_X^p$ is the $p$-th Du Bois complex of $X$; see Sections 1 and 4 in \cite{PP25a} for details.

We often compare these notions with similar ones coming from the study of intersection cohomology. We set
$$
\mathit{IH}^{p,q}(X) : = \gr^F_{-p}\mathit{IH}^{p+q}(X,\C)=  \HH^q(X,I\DB_X^p).
$$
Here 
$$\mathit{IH}^{i + n} (X, \C) = \HH^i (X, \IC_X)$$
are the intersection cohomology groups of $X$, and $\IC_X$ is its $\Q$-intersection complex, underlying the \emph{intersection complex Hodge module} $\IC^H_X$ from \cite{Saito90}; moreover, $I\DB_X^p$ are the \emph{intersection Du Bois complexes} of $X$, i.e. the graded pieces of the filtered de Rham complex of $\IC^H_X$.  The \textit{intersection Hodge numbers} of $X$ are defined as 
$$I\underline h^{p,q}(X)  := \dim_\C \mathit{IH}^{p,q}(X).$$
For all this, see Sections 3 and 4 in \cite{PP25a}.

There is a distinguished triangle 
\begin{equation}\label{eqn:RHM-object}
\K_X^\bullet \longrightarrow \Q_X^H[n] \longrightarrow \IC_X^H\xrightarrow{+1}  
\end{equation}
defined by the canonical map from the trivial Hodge module to the intersection complex. The object $\K_X^\bullet$ in the derived category of mixed Hodge modules on $X$ is studied, and called the \emph{RHM-defect object of $X$}, in Section 6 of \cite{PP25a}.
Passing to cohomology  induces natural maps
$$H^{p, q} (X)  \to \mathit{IH}^{p,q}(X),$$
which in turn are Hodge pieces of the natural topological map 
$$H^{p+q} (X, \C) \to \mathit{IH}^{p+q} (X, \C).$$

Intersection cohomology is well known to satisfy Poincar\'e duality and the weak Lefschetz theorem \cite{GM88}. On the other hand, the extent 
to which is true for singular cohomology is dictated by the so-called \emph{local cohomological defect} 
$${\rm lcdef} (X) : = {\rm lcd} (X, Y) - \codim_Y X,$$
where $Y$ is any smooth variety containing $X$ (locally), and ${\rm lcd} (X, Y)$ is the local cohomological dimension of $X$ in $Y$.
See Section 2 in \cite{PP25a} for a discussion of this invariant, and Section 8 in \emph{loc. cit.} for the weak Lefschetz theorem for a hyperplane section of a projective $X$ that depends on this invariant. 

In particular, a variety $X$ with $\lcdef (X) = 0$ satisfies the exact same weak Lefschetz theorem as in the smooth case. This class includes Cohen-Macaulay varieties of dimension up to $3$ and rational homology manifolds; see the references in Section 2 of \cite{PP25a}. Characterizations of $\lcdef (X)$ for higher dimensional varieties, in terms of higher cohomological invariants, can be found in the Appendix of \cite{PP25a}.

For the background on Hodge modules needed in this paper, please see Section 5 in \cite{PP25a}.

\subsection{Lefschetz (1,1)-theorem for singular varieties}\label{scn:(1,1)}
The purpose of this section is to address the following Weil divisor version of the Lefschetz $(1,1)$-theorem on normal varieties. This is a useful tool
towards the proof of Theorem \ref{thm:Q-factoriality-defect}, but also a result of general interest, containing in particular a comparison between various notions of equivalence of Weil divisors.\footnote{Another extension of the Lefschetz $(1,1)$-theorem to normal varieties was found by Biswas-Srinivas \cite{BS00}, who describe 
${\rm NS} (X)$, partly in terms of the mixed Hodge structure on singular cohomology.} We explain the proof of (i) in this section, and the proof of (ii) in the next.

\begin{thm}\label{thm:Lefschetz-(1,1)} 
Let $X$ be a normal projective variety. Then: 

\noindent
(i)  There exists a cycle class morphism
$$
{\rm cl}_\Q \colon \Div_\Q(X)\to \mathit{IH}^2(X,\Q)\cap \mathit{IH}^{1,1}(X)
$$
which is surjective. 

\noindent
(ii)  The kernel of the morphism in (i) is $\mathrm{Hom}^1_\Q(X) = \mathrm{Alg}^1_\Q(X)$. In other words, on normal varieties, algebraic equivalence coincides with the homological equivalence for $\Q$-Weil divisors. Furthermore, if $X$ is $\Q$-factorial, then they also coincide with numerical equivalence:
$$
\mathrm{Alg}^1_\Q(X)=\mathrm{Hom}^1_\Q(X)=\mathrm{Num}^1_\Q(X).
$$
\end{thm}

%When $X$ is projective of dimension $n$, the group $\mathrm{Hom}^1_\Q(X)$ is, by definition, the kernel of the cycle class map to the singular homology of $X$:
%$$
%\Div_\Q(X)\to H_{2n-2}(X,\Q).
%$$

We emphasize that much of this result is not new.  The statement in (i) is already known in a different form, due to Jannsen \cite{Jannsen90}, who reduced a singular version of the Hodge conjecture to the case of smooth varieties; our contribution to (i) is to express the right hand side in terms of the Hodge structure on Borel-Moore homology, and then intersection cohomology, as in Lemma \ref{lem: lowest weight of H2 Borel Moore}. This turns out to be very useful throughout. The method we use leads to the statement in part (ii),  as well as to the proof of Theorem \ref{thm:Q-factoriality-defect}. The first statement in (ii) is familiar to experts;  for instance, it can be deduced from the statement of \cite[Theorem 7.3]{BO74}, but there are also more elementary proofs. 

Turning to details, the cycle class map for smooth projective varieties is generalized in various contexts in the literature. We recall a version for arbitrary irreducible varieties. Given an irreducible variety $X$ of dimension $n$, there is a cycle class map with rational coefficients
$$
\cl_d\tensor \Q: Z_d(X)\tensor \Q\to H^{BM}_{2d}(X,\Q),
$$
where $Z_d(X)$ is a free abelian group of algebraic cycles of dimension $d$ in $X$, and $H^{BM}_{\bullet}(X,\Q)$ is the Borel-Moore homology
of $X$; see \cite[Section 19.1]{Fulton98}. The group of algebraic cycles of dimension $d$ with zero homology class is denoted by
$$
\mathrm{Hom}_\Q^{n-d}(X):=\ker(\mathrm{cl}_d\tensor \Q).
$$
It is well known that this cycle class map commutes with pushforward for proper morphisms.

Borel-Moore homology carries a mixed Hodge structure coming from Deligne's theory; see for instance \cite[6.24]{PS08}. On the other hand, 
it is naturally identified with the hypercohomology of the Verdier dualizing sheaf
$$
H^{BM}_{2d}(X,\Q)=\HH^{-2d}(X,\dual \Q_X),
$$
and thus inherits a mixed Hodge structure coming from Saito's theory, as in \cite[\S4.5]{Saito90}. These two mixed Hodge structures are known to be equivalent; see \emph{loc. cit.} or \cite{Saito00}. Consequently, the image of $\cl_d\tensor \Q$ is contained in
\begin{equation}
\label{eqn: cycle class target}
W_{-2d}\HH^{-2d}(X,\dual\Q_X)\cap F^{-d}\HH^{-2d}(X,\dual \C_X).
\end{equation}
(See, for example, \cite[Section 7.1.3]{PS08}). It is worth noting that when $X$ is a smooth projective variety, \eqref{eqn: cycle class target} is naturally isomorphic to the group of Hodge classes
$$
H^{2n-2d}(X,\Q)\cap H^{n-d,n-d}(X)
$$
by Poincar\'e duality, and the Hodge conjecture asserts the surjectivity of $\cl_d\tensor \Q$ onto this group.

For a normal projective variety $X$ and $d=n-1$, we have the following key identification.

\begin{lem}
\label{lem: lowest weight of H2 Borel Moore}
Let $X$ be a normal projective variety of dimension $n$. Then we have a canonical isomorphism of pure $\Q$-Hodge structures of weight $2$:
$$
\mathit{IH}^2(X,\Q)=\left(W_{-2n + 2} \HH^{-2n+2}(X, \dual \Q_X)\right) (-n).
$$
In particular, when $d=n-1$, \eqref{eqn: cycle class target} is canonically isomorphic to
$$
\mathit{IH}^2(X,\Q)\cap\mathit{IH}^{1,1}(X)
$$
where $\mathit{IH}^{1,1}(X)$ is the $(1,1)$-component of $\mathit{IH}^2(X,\Q)$.
\end{lem}

\begin{proof}
The dual of the distinguished triangle ($\ref{eqn:RHM-object}$) is
$$
\dual\IC_X^H\to(\dual\Q_X^H)[-n]\to\dual \K_X^\bullet\xrightarrow{+1}.
$$
Since we have the canonical polarization $\IC^H_X(n)=\dual\IC^H_X$, this induces a long exact sequence of mixed Hodge structures
$$
\cdots \to\HH^{-n+1}(X,\dual \K_X^\bullet)\to\HH^{-n+2}(X,\IC_X)(n)\to \HH^{-2n+2}(X, \dual \Q_X)\to \HH^{-n+2}(X,\dual \K_X^\bullet)\to \cdots.
$$
As $X$ is normal, the (perverse) cohomologies of $\K^\bullet_X$ are supported on a subvariety of dimension $\le n-2$. Moreover, since $\dual\K^\bullet_X\in D^{\ge 0}{\rm MHM}(X)$ and $\dual\K^\bullet_X$ is of weight $\ge -n+1$ by \cite[Proposition  6.4]{PP25a},  we have $\HH^{-n+1}(X,\dual \K_X^\bullet)=0$ (see \cite[Proposition 8.1.42]{HTT}), and $\HH^{-n+2}(X,\dual \K_X^\bullet)$ is a mixed Hodge structure of weight $\ge -2n+3$.

In conclusion, we have a natural inclusion
$$
\mathit{IH}^2(X,\Q)= \HH^{-n+2}(X,\IC_X)\hookrightarrow \HH^{-2n+2}(X,(\dual \Q_X))(-n)
$$
whose quotient is a mixed Hodge structure of weight $\ge 3$. This completes the proof of the first statement.
The last statement follows from the definition of a pure Hodge structure.
\end{proof}

Combining the general construction with Lemma \ref{lem: lowest weight of H2 Borel Moore}, when $X$ is a normal projective variety, we obtain a cycle class morphism
for Weil divisors on $X$, as in Theorem \ref{thm:Lefschetz-(1,1)}(i).

A homological version of the generalized Hodge conjecture \cite[Conjecture 7.12]{PS08} states that $\cl_d\tensor \Q$ surjects onto \eqref{eqn: cycle class target}, and using resolution of singularities Jannsen \cite[Theorem 7.9]{Jannsen90} proved that this version is true if it is true for smooth projective varieties. In particular, this holds when $d=n-1$, by the classical Lefschetz $(1,1)$-theorem.

Therefore, the surjectivity of $\cl_\Q$ in Theorem \ref{thm:Lefschetz-(1,1)} is an immediate consequence.

\begin{rmk}
Using the mixed Hodge module formalism, Saito \cite[4.5.18]{Saito90} constructed a cycle class map
$$
Z_d(X)\tensor \Q\to \mathrm{Hom}_{D^b{\rm MHM}({\rm pt})}(\Q_{\rm pt}^H, (a_X)_*(\dual\Q^H_X)(-d)[-2d]),
$$
where $a_X:X\to \mathrm{pt}$. Taking 0-th cohomologies on the right hand side, this induces a map
$$
Z_d(X)\tensor \Q\to \mathrm{Hom}_{{\mathrm {MHS}}}(\Q^H, \HH^{-2d}(X, \dual \Q^H_X(-d))),
$$
where $\Q^H$ is a trivial pure $\Q$-Hodge structure of weight $0$ and the right hand side is the group of morphisms of mixed Hodge structures. This group is naturally isomorphic to \eqref{eqn: cycle class target}, which offers another way to construct the cycle class map 
we consider.
\end{rmk}

\subsection{Algebraic, homological, and numerical equivalence of Weil divisors}
Following the terminology in Fulton \cite{Fulton98}, we denote by $\mathrm{Alg}^1_\Q(X)$ the group of $\Q$-Weil divisors algebraically equivalent to zero, and by $\mathrm{Num}^1_\Q(X)$ the group of $\Q$-Weil divisors numerically equivalent to zero. Here, a $\Q$-Weil divisor $\alpha$ is numerically equivalent to zero if
$$
\int_X P \cap \alpha=0
$$
for all polynomials $P$ in Chern classes of vector bundles on $X$. For a projective variety $X$, it is proven in \cite[Section 19.1]{Fulton98} that there are inclusions
$$
\mathrm{Alg}^1_\Q(X)\subset \mathrm{Hom}^1_\Q(X)\subset \mathrm{Num}^1_\Q(X)\subset \Div_\Q(X).
$$

Beyond the surjectivity of the cycle class morphism, the classical Lefschetz $(1,1)$-theorem for smooth projective varieties explains that the algebraic, homological, and numerical equivalence relations coincide. This is generalized by Theorem \ref{thm:Lefschetz-(1,1)}(ii), stating that for a normal projective $X$ 
we have $\mathrm{Alg}^1_\Q(X)=\mathrm{Hom}^1_\Q(X)$ (which as we mentioned is already known), while if in addition $X$ is $\mathbb{Q}$-factorial, they also coincide with  $\mathrm{Num}^1_\Q(X)$.

The proof uses a desingularization $\mu \colon \widetilde X\to X$ and establishes that the proper pushforward map
$$
\mu_*:\mathrm{Hom}^1_\Q(\widetilde X)\to \mathrm{Hom}^1_\Q(X).
$$
is an isomorphism. From this, it follows easily that $\mathrm{Alg}^1_\Q(X)=\mathrm{Hom}^1_\Q(X)$, as algebraic equivalence is preserved under proper pushforward.

For this purpose, we compare the (intersection) cohomology of the two spaces $X$ and $\widetilde X$. Note that there exists a natural surjection morphism
$$
H^2(\widetilde X,\Q)\to \mathit{IH}^2(X,\Q)
$$
of pure $\Q$-Hodge structures of weight $2$. Indeed, there exists a pushforward map 
$$\mu_*:\HH^{-2n+2}(\widetilde X,\dual \Q_{\widetilde X})\to\HH^{-2n+2}(X,\dual \Q_X)$$ 
of mixed Hodge structures on Borel-Moore homology, induced by the adjunction morphism $\mu_*\dual \Q_{\widetilde X}\to \dual \Q_X$. By Poincaré duality, the hypercohomology $\HH^{-2n+2}(\widetilde X,\dual \Q_{\widetilde X})(-n)$ is equal to the pure Hodge structure $H^2(\widetilde X,\Q)$ of weight $2$. By Lemma \ref{lem: lowest weight of H2 Borel Moore}, this implies that this pushforward map on hypercohomology gives the natural morphism $H^2(\widetilde X,\Q)\to \mathit{IH}^2(X,\Q)$ of pure Hodge structures of weight $2$, and also the induced pushforward map
$$
\mu_*:H^2(\widetilde X,\Q)\cap H^{1,1}(\widetilde X)\to \mathit{IH}^2(X,\Q)\cap \mathit{IH}^{1,1}(X).
$$
We describe the kernel of this morphism in the following:

\begin{prop}
\label{prop: kernel of H2 to IH2}
Let $X$ be a normal projective variety of dimension $n$. Let $\mu:\widetilde X\to X$ be a resolution of singularities, and denote by  $\left\{E_i\right\}_{i\in I}$ 
the set of exceptional divisors  in $\widetilde X$. Then
$$
\ker\big(H^2(\widetilde X,\Q)\to \mathit{IH}^{2}(X,\Q)\big)=\bigoplus_{i\in I}\Q\cdot \cl_\Q(E_i),
$$
that is the kernel of the natural surjection $H^2(\widetilde X,\Q)\to \mathit{IH}^{2}(X,\Q)$ has the set $\left\{\cl_\Q(E_i)\right\}_{i\in I}$ of cycle classes of exceptional divisors as a basis.
\end{prop}

The proof of this proposition uses the next two lemmas. The first states that the classes of exceptional divisors are linearly independent, and the second states that these classes generate the kernel of the natural morphism appearing in the proposition. The first lemma is well-known to experts and we formulate this for normal analytic varieties, as it is used to study a local analogue of the Lefschetz (1,1)-theorem. Note that the cycle class map $\cl_\Q$ is defined for arbitrary analytic varieties and coincides with the first Chern class map for divisors on smooth analytic varieties.

\begin{lem}
\label{lem: linear independence of exceptional divisors}
Let $X$ be a normal analytic variety, and a projective morphism $\mu:\widetilde X\to X$ be a resolution of singularities with $\left\{E_i\right\}_{i\in I}$ the set of exceptional divisors in $\widetilde X$. If
$$
\sum_{i\in I}a_i\cdot \cl_\Q(E_i)= 0,
$$
for a finite $\Q$-linear combination of classes of exceptional divisors, then $a_i=0$ for all $i\in I$.
\end{lem}

\begin{proof}
We proceed by induction on $n = \dim X$. When $X$ is a normal surface, this is in \cite[p.367]{Grauert62}. Supposing now that the statement is true when the dimension is $n-1$, we prove it in dimension $n$.

For every point of $x\in X$, it suffices to prove that the coefficients of exceptional divisor over a small neighborhood of $x$ is zero. Hence, we assume that $x\in X$ is the germ of a normal analytic variety.
We argue by contradiction. Let
$$J:=\left\{i\in I | a_i\neq 0\right\}$$ 
and suppose $J$ is nonempty. Note first that it suffices to consider the case when $\mu(E_j)$ is supported at a closed point of $X$ for all $j\in J$. Indeed, say $a_j\neq 0$, with $\dim\mu(E_j)\ge 1$. Taking a general hyperplane section $H\subset X$, the restriction of the class of  the $\Q$-linear combination to $\mu^{-1}(H)$ is still zero:
$$
\sum_{i\in I}a_i\cdot \cl_\Q(E_i|_{\mu^{-1}(H)})= 0.
$$
Since $\mu|_{\mu^{-1}(H)}\colon \mu^{-1}(H)\to H$ is a resolution of singularities of the normal variety $H$ of dimension $n-1$, with exceptional divisors $E_i|_{\mu^{-1}(H)}$, we have $a_j=0$ by the induction hypothesis. This is a contradiction.

Assume now that $\mu(E_j)$ is a point for all $j\in J$. Take a general hyperplane section $\widetilde H\subset \widetilde X$. As above, we have
$$
\sum_{j\in J}a_j\cdot \cl_\Q(E_j|_{\widetilde H})= 0.
$$
Then the restriction morphism $\mu|_{\widetilde H}\colon \widetilde H\to \mu(\widetilde H)$ factors through the normalization of $\mu(\widetilde H)$, and $E_j|_{\widetilde H}$ are exceptional divisors for all $j\in J$. By the induction hypothesis, $a_j=0$ for all $j\in J$. This is a contradiction, which completes the proof.
\end{proof}

\begin{lem}
\label{lem: difference of h2 and IH2}
In the setting of Proposition \ref{prop: kernel of H2 to IH2}, we have
$$
h^2(\widetilde X)-  Ih^2(X)=|I|.
$$
\end{lem}

\begin{proof}
By Poincaré duality, it suffices to show that
$$
h^{2n-2}(\widetilde X)-  Ih^{2n-2}(X)=|I|.
$$
By the Decomposition Theorem, we have
$$
R\mu_*\Q_{\widetilde X}[n] \simeq \IC_X[0]\oplus \M^\bullet
$$
where $\M^\bullet$ is supported inside the locus $Z\subset X$ where $\mu$ is not an isomorphism. Consequently, we have an isomorphism of hypercohomologies
$$
H^{2n-2}(\widetilde X,\Q)\simeq \mathit{IH}^{2n-2}(X,\Q)\oplus \HH^{n-2}(Z,\M^\bullet).
$$

Denote by $\iota:Z\hookrightarrow X$ the closed embedding. Note that $\dim Z\le n-2$. By the proper base change theorem, we have
$$
R\mu_*\Q_{\mu^{-1}(Z)}[n]=\iota^*\IC_X[0]\oplus \M^\bullet.
$$
Since $\iota^*\IC_X[0]\in {}^p\!D^{\le -1}_c(Z,\Q)$, we have
$$
\HH^{n-2}(Z,\iota^*\IC_X)=0.
$$
(See e.g. \cite[Propositions 8.1.42, 8.2.5]{HTT}.) Therefore we obtain
$$
\dim_\Q\HH^{n-2}(Z,\M^\bullet)=h^{2n-2}(\mu^{-1}(Z),\Q)=|I|,
$$
which completes the proof.
\end{proof}

Next, we prove Proposition \ref{prop: kernel of H2 to IH2} and Theorem  \ref{thm:Lefschetz-(1,1)}(ii).

\begin{proof}[Proof of Proposition \ref{prop: kernel of H2 to IH2}]
Exceptional divisors map to zero by the proper pushforward map, so the classes of exceptional divisors 
$\cl_\Q(E_i)$ are contained in the kernel of the natural surjective morphism $H^2(\widetilde X,\Q)\to \mathit{IH}^2(X,\Q)$.
Since Lemma \ref{lem: linear independence of exceptional divisors} proves that these classes are linearly independent, we conclude from Lemma \ref{lem: difference of h2 and IH2} that they form a basis of the kernel.
\end{proof}

\begin{proof}[Proof of Theorem  \ref{thm:Lefschetz-(1,1)}(ii)]
Recall the discussion preceding Proposition \ref{prop: kernel of H2 to IH2}, which leads to a  commutative diagram of proper pushforward and cycle class morphisms:
\begin{equation*}
\xymatrix{
{\Div_\Q(\widetilde X)}\ar[d]_{\cl_\Q^{\widetilde X}}\ar[r]^-{\mu_*}& {\Div_\Q(X)}\ar[d]_{\cl_\Q^X}\\
{H^2(\widetilde X,\Q)\cap H^{1,1}(\widetilde X)}\ar[r]^-{\mu_*}&{\mathit{IH}^2(X,\Q)\cap \mathit{IH}^{1,1}(X)}
}
\end{equation*}
Note that all the maps are surjective; for the right-most, we use Theorem \ref{thm:Lefschetz-(1,1)}(i). Additionally,
$$
\ker(\mu_*:\Div_\Q(\widetilde X)\to \Div_\Q(X))=\bigoplus_{i\in I}\Q\cdot E_i
$$
where $\left\{E_i\right\}_{i\in I}$ is the set of $\mu$-exceptional divisors  in $\widetilde X$, and
$$
\ker(\mu_*:H^2(\widetilde X,\Q)\cap H^{1,1}(\widetilde X)\to \mathit{IH}^2(X,\Q)\cap \mathit{IH}^{1,1}(X))=\bigoplus_{i\in I}\Q\cdot \cl_\Q(E_i)
$$
by Proposition \ref{prop: kernel of H2 to IH2}. Indeed, for the second equality, it is easy to check that 
$$
\gr^F_0H^2(\widetilde X,\C)=\gr^F_0\mathit{IH}^2(X,\C)
$$
(see \cite[Lemma 4.2]{Park23} for details), which implies that
$$
\ker(H^2(\widetilde X,\Q)\cap H^{1,1}(\widetilde X)\to \mathit{IH}^2(X,\Q)\cap \mathit{IH}^{1,1}(X))=\ker(H^2(\widetilde X,\Q)\to \mathit{IH}^2(X,\Q)).
$$
Consequently, applying the Snake Lemma, we deduce
\begin{equation}
\label{eqn: kernel of cycle class map is desingularization invariant}
\ker(\Div_\Q(\widetilde X)\to H^2(\widetilde X,\Q)\cap H^{1,1}(\widetilde X))=\ker(\Div_\Q(X)\to \mathit{IH}^2(X,\Q)\cap \mathit{IH}^{1,1}(X)),
\end{equation}
that is, the kernel of the cycle class morphism is invariant under a resolution of singularities. This implies that
$$
\mu_*: \mathrm{Hom}^1_\Q(\widetilde X)\to \mathrm{Hom}^1_\Q(X)
$$
is an isomorphism. From the classical Lefschetz $(1,1)$-theorem, we have $\mathrm{Hom}^1_\Q(\widetilde X)=\mathrm{Alg}^1_\Q(\widetilde X)$. Additionally, algebraic equivalence is preserved under the proper pushforward $\mu_*$; see \cite[Proposition 10.3]{Fulton98}. Therefore, $\mathrm{Hom}^1_\Q(X)=\mathrm{Alg}^1_\Q(X)$.

Next, assume $X$ is $\Q$-factorial, or equivalently, $\CDiv_\Q(X)=\Div_\Q(X)$. Consider the composition
$$
\CDiv_\Q(X)\to H^2(X,\Q)\to \mathit{IH}^2(X,\Q)
$$
where the first map is the first Chern class map. This composition is the cycle class map in Theorem \ref{thm:Lefschetz-(1,1)}(i) restricted to $\Q$-Cartier divisors; see \eqref{eqn: divisor class commuting diagram} below for a more detailed explanation. Thus by the assumption, $\CDiv_\Q(X)$ maps surjectively onto the $(1,1)$-component of $\mathit{IH}^2(X,\Q)$.

Let $\eta\in H^2(X,\Q)$ be the class of an ample divisor. Then we have the commutative diagram:
\begin{equation*}
\xymatrix{
{\CDiv_\Q(X)}\ar[r]&{H^2(X,\Q)}\ar[d]_{\eta^{n-2}\cdot}\ar[r]&{\mathit{IH}^2(X,\Q)}\ar[d]_{\eta^{n-2}\cdot} \\
{}&{H^{2n-2}(X,\Q)}\ar[r]&{\mathit{IH}^{2n-2}(X,\Q)}
}
\end{equation*}
where the vertical maps are the cup products with the class $\eta^{n-2}$. In particular, the vertical maps are morphisms of (mixed) Hodge structures up to Tate twist, and the right-most is an isomorphism by the Hard Lefschetz theorem for intersection cohomology. The dual of Lemma \ref{lem: lowest weight of H2 Borel Moore} says that $\mathit{IH}^{2n-2}(X,\Q)$ is the top weight graded piece of $H^{2n-2}(X,\Q)$:
$$
\mathit{IH}^{2n-2}(X,\Q)=\gr^W_{2n-2}H^{2n-2}(X,\Q).
$$
Thus, $\eta^{n-2}\cdot \CDiv_\Q(X)$ maps surjectively onto the $(n-1,n-1)$-component of $\gr^W_{2n-2}H^{2n-2}(X,\Q)$.

On the other hand, $\Div_\Q(X)$ maps surjectively onto the $(-n+1,-n+1)$-component of 
$$
\gr^W_{-2n+2}H_{2n-2}(X,\Q)
$$
by Theorem \ref{thm:Lefschetz-(1,1)}(i). Using the perfect pairing
$$
\gr^W_{-2n+2}H_{2n-2}(X,\Q)\times \gr^W_{2n-2}H^{2n-2}(X,\Q)\to \Q,
$$
it then follows that $\mathrm{Hom}^1_\Q(X)=\mathrm{Num}^1_\Q(X)$.
\end{proof}

\begin{rmk}
In fact, one can prove $\mathrm{Alg}^1_\Q(X)=\mathrm{Hom}^1_\Q(X)=\mathrm{Num}^1_\Q(X)$ when
$$
\dim_\Q(\mathit{IH}^2(X,\Q)\cap \mathit{IH}^{1,1}(X))=\dim_\Q (\ker(H^2(X,\Q)\to H^2(X,\O_X)))
$$
using results in the next section. This condition is potentially weaker than $\Q$-factoriality, as seen in Theorem \ref{thm:Q-factoriality-defect}.
\end{rmk}

\iffalse

\begin{question}[Integer coefficients] For a smooth projective variety $X$, we already have $\mathrm{Alg}^1(X)=\mathrm{Hom}^1(X)$ for $\Z$-divisors.
Additionally, we have
$$
\mathrm{Alg}^1_\tau(X)/\mathrm{Alg}^1(X)=H^2(X,\Z)_{\mathrm{tors}}
$$
where $\mathrm{Alg}^1_\tau(X):=\Div(X)\cap\mathrm{Alg}^1_\Q(X)$. We ask if $\mathrm{Alg}^1(X)=\mathrm{Hom}^1(X)$ is still true for a normal projective variety $X$, and if there exists an integral cohomological description of $\mathrm{Alg}^1_\tau(X)/\mathrm{Alg}^1(X)$.
\end{question}

\begin{question}[Positive characteristic] For a smooth projective variety $X$ over an algebraically closed field of arbitrary characteristic, Matsusaka \cite{Matsusaka57} proved that algebraic and numerical equivalence coincide for $\Q$-divisors on $X$, and furthermore, both coincide with homological equivalence for any Weil cohomology theory. When $X$ is a normal projective variety over an algebraically closed field of arbitrary characteristic, we ask whether Theorem \ref{thm: equivalence relations} also holds for some Weil cohomology theory.
\end{question}

\fi

For later use, we also record the following result regarding varieties with normal or rational singularities.

\begin{lem}
\label{lem: H1 and H2 of normal variety}
Let $X$ be a normal projective variety of dimension $n$, and let $\mu\colon \widetilde X\to X$ be a resolution of singularities. Then, $h^{2n-1}(X)= Ih ^{2n-1}(X)= Ih^{1}(X)$ and we have canonical isomorphisms
$$
\gr^F_0H^1(X,\C) \simeq H^1(X,\O_X) \quad \mathrm{and}\quad H^1(\widetilde X,\Q)\simeq \mathit{IH}^1(X,\Q).
$$
Furthermore, if $X$ has rational singularities, then $h^{2n-2}(X)= Ih^{2n-2}(X)= Ih^2(X)$.
\end{lem}

\begin{proof}
Recall from \cite[Section 6]{PP25a} that the ${\rm RHM}$-defect object satisfies $\K^\bullet_X\in D^{\le0}{\rm MHM}(X)$, and has cohomologies supported in $\Sing(X)$. 
Using \cite[Proposition 8.1.42]{HTT}, this implies $\HH^{i}(X,\K^\bullet_X)=0$ for $i\ge 2n-1$, which in turn immediately gives $h^{2n-1}(X)=Ih^{2n-1}(X)$
by passing to cohomology in the exact triangle ($\ref{eqn:RHM-object}$).

Likewise, if $X$ has rational singularities, then $\K^\bullet_X$ has cohomologies supported in a subvariety of dimension at most $\le n-3$, by case $k=0$ in \cite[Corollary 7.5]{PP25a}. This implies the vanishing $\HH^{i}(X,\K^\bullet_X)=0$ for $i\ge 2n-2$, hence $h^{2n-2}(X)= Ih^{2n-2}(X)$.

Note now that the canonical morphism $\C_X\to \DB_X^0$ factors through $\O_X$. Since the map $H^1(X,\C)\to \HH^1(X,\DB^0_X)$ is surjective by the degeneration of 
the Hodge-to-de Rham spectral sequence in the singular case, the map $H^1(X,\O_X)\to \HH^1(X,\DB^0_X)$ 
is surjective as well. Additionally, this map is also injective using a simple spectral sequence argument, since $\H^0(\DB_X^0)=\O_X$ by the normality of $X$ (see e.g. \cite[Proposition 5.2]{Saito00}). Therefore, we have a canonical isomorphism
$$
H^1(X,\O_X) \simeq \HH^1(X,\DB^0_X).
$$

Lastly, a canonical isomorphism $H^1(\widetilde X,\Q) \simeq \mathit{IH}^1(X,\Q)$ is obtained from the canonical pushforward morphism
$$
H^1(\widetilde X,\Q)=\HH^{-2n+1}(\widetilde X,(\dual\Q^H_{\widetilde X})(-n))\xrightarrow{\mu_*} \HH^{-2n+1}(X,(\dual\Q^H_X)(-n))
$$
and the canonical isomorphism
$$
\mathit{IH}^1(X,\Q)=\HH^{-2n+1}(X,(\dual\Q^H_X)(-n))
$$
obtained by dualizing the argument in the proof of the equality $h^{2n-1}(X)=Ih^{2n-1}(X)$. By the Decomposition Theorem, it is easy to see that $H^1(\widetilde X,\Q)\to \mathit{IH}^1(X,\Q)$ is surjective, and the two spaces have the same dimension (see e.g. \cite[Lemma 4.2]{Park23}). Hence the morphism above is an isomorphism.
\end{proof}

\subsection{The \texorpdfstring{$\Q$}{Q}-factoriality defect}
For a normal variety $X$, we consider the \textit{defect of $\Q$-factoriality}
$$
\sigma(X):=\dim_\Q \Div_\Q(X)/\CDiv_\Q(X).
$$
We have $\sigma(X)=0$ if and only if $X$ is $\Q$-factorial.

The goal of this section is to prove Theorem \ref{thm:Q-factoriality-defect}, which, as explained in the Introduction, extends previously known results in the literature 
to a general formula for $\sigma (X)$. This theorem is about the comparison between $\Div_\Q(X)$ and $\CDiv_\Q(X)$; we approach this by building on our understanding of 
$\mathrm{Alg}_\Q^1(X)$ in $\Div_\Q(X)$ in the previous section, and establishing a result that compares $\mathrm{Alg}_\Q^1(X)$ with $\CDiv_\Q(X)$.

\begin{prop}
\label{prop: algebraic equivalence and Cartier divisors}
Let $X$ be a normal projective variety of dimension $n$. Then
$$
\dim_\Q \frac{\mathrm{Alg}^1_\Q(X)}{\mathrm{Alg}^1_\Q(X)\cap \CDiv_\Q(X)}= \begin{cases}
0 &\text{if}\;\; h^1(X)=h^{2n-1}(X)\\
\infty &\text{otherwise}.
\end{cases}
$$
In particular, if $h^1(X)= h^{2n-1}(X)$, then every deformation of a $\Q$-Cartier divisor is $\Q$-Cartier.
\end{prop}
\begin{proof}
Consider the exponential exact sequence:
$$
0\to \Z_X\to \O_X\to \O_X^*\to 1.
$$
This sequence is functorial, and thus, we have the commuting long exact cohomology sequences
\begin{equation}
\label{eqn: exponential LES commuting}
\xymatrix{
{H^1(X,\Z)}\ar[r]\ar[d]_{\mu^*}& {H^1(X,\O_X)}\ar[d]_{\mu^*}\ar[r]&{\Pic (X)}\ar[d]_{\mu^*}\ar[r]^-{c_1}&{H^2(X,\Z)}\ar[d]_{\mu^*}\ar[r]&{H^2(X,\O_X)}\ar[d]_{\mu^*} \\
{H^1(\widetilde X,\Z)}\ar[r]^-{}&{H^1(\widetilde X, \O_{\widetilde X})}\ar[r]&{\Pic (\widetilde X)}\ar[r]^-{c_1}&{H^2(\widetilde X,\Z)}\ar[r]&{H^2(\widetilde X,\O_{\widetilde X})}
}
\end{equation}
where $\mu\colon \widetilde X\to X$ is a resolution of singularities and $c_1$ is the first Chern class map. There exists a natural surjective map
$$
\CDiv(X)\to \Pic (X)
$$
which maps $D\in \CDiv(X)$ to $\O_X(D)\in \Pic(X)$, and this can be composed with the first Chern class map to obtain $\CDiv(X)\to H^2(X,\Z)$. Taking rational coefficients, we have a natural map
$$
\tilde c_1^X:\CDiv_\Q(X)\to H^2(X,\Q).
$$
(In this proof we add superscripts, to emphasize the variety we are considering.)
Accordingly, from \eqref{eqn: exponential LES commuting} we obtain the commutative diagram 
\begin{equation}
\label{eqn: exponential LES commuting 2}
\xymatrix{
{\CDiv_\Q(X)}\ar[d]_{\mu^*}\ar[r]^-{\tilde c_1^X}&{H^2(X,\Q)}\ar[d]_{\mu^*}\ar[r]&{H^2(X,\O_X)}\ar[d]_{\mu^*} \\
{\CDiv_\Q(\widetilde X)}\ar[r]^-{\tilde c_1^{\widetilde X}}&{H^2(\widetilde X,\Q)}\ar[r]&{H^2(\widetilde X,\O_{\widetilde X})}
}
\end{equation}
where the horizontal rows are exact. Notice that the map $\tilde c_1^{\widetilde X}$ for the smooth projective variety $\widetilde X$ agrees with the cycle class morphism $\cl_\Q^{\widetilde X}$ for $\widetilde X$ in Theorem \ref{thm:Lefschetz-(1,1)}(i).

Denote $V_X:=\ker(H^2(X,\Q)\to H^2(X,\O_X))$. Then we have the induced pullback map
$$
\mu^*:V_X\to V_{\widetilde X}.
$$
Since $H^2(\widetilde X,\O_{\widetilde X})=\gr^F_0 H^2(\widetilde X,\C)$, we have $V_{\widetilde X}=H^2(\widetilde X,\Q)\cap H^{1,1}(\widetilde X)$. By \eqref{eqn: exponential LES commuting 2} and the commutativity of the proper pushforward functor and the cycle class morphism, we have the commutative diagram:
\begin{equation}
\label{eqn: divisor class commuting diagram}
\xymatrix{
{\CDiv_\Q(X)}\ar[d]_{\tilde c_1^X}\ar[r]^-{\mu^*}&{\Div_\Q(\widetilde X)}\ar[d]_{\tilde c_1^{\widetilde X}= \cl_\Q^{\widetilde X}}\ar[r]^-{\mu_*}& {\Div_\Q(X)}\ar[d]_{\cl_\Q^X}\\
{V_X}\ar[r]^-{\mu^*}&{H^2(\widetilde X,\Q)\cap H^{1,1}(\widetilde X)}\ar[r]^-{\mu_*}&{\mathit{IH}^2(X,\Q)\cap \mathit{IH}^{1,1}(X)}
}
\end{equation}
where all the vertical columns are surjective; for the right-most, we use Theorem \ref{thm:Lefschetz-(1,1)}(i). The right square appears in the proof of Theorem 
\ref{thm:Lefschetz-(1,1)}(ii), where we showed that
$$
\ker(\cl_\Q^{\widetilde X})=\ker(\cl_\Q^{X})=\mathrm{Alg}^1_\Q(X).
$$

Next we focus on the outer square of \eqref{eqn: divisor class commuting diagram}:
\begin{equation}
\label{eqn: last divisor class commuting diagram}
\xymatrix{
{\CDiv_\Q(X)}\ar[d]_{\tilde c_1^X}\ar[r]^-{\mu_*\circ\mu^*}& {\Div_\Q(X)}\ar[d]_{\cl_\Q^X}\\
{V_X}\ar[r]^-{\mu_*\circ\mu^*}&{\mathit{IH}^2(X,\Q)\cap \mathit{IH}^{1,1}(X)}
}
\end{equation}
It is straightforward to check that the first row is the natural inclusion $\CDiv_\Q(X)\hookrightarrow \Div_\Q(X)$. Notice that
$$
V_X\subset \ker(H^2(X,\Q)\to \gr^F_0H^2(X,\C)),
$$
because of the factorization
$$
\Q_X\to\O_X\to \DB_X^0,
$$
and the fact that $\gr^F_0H^2(X,\C)=\HH^2(X,\DB_X^0)$. The bottom row of \eqref{eqn: last divisor class commuting diagram} is induced by the morphism of mixed Hodge structures $H^2(X,\Q)\to \mathit{IH}^2(X,\Q)$.  This morphism factors through $\gr^W_{2}H^2(X,\Q)$, as a surjection followed by an injection
$$
H^2(X,\Q)\twoheadrightarrow \gr^W_{2}H^2(X,\Q) \hookrightarrow \mathit{IH}^2(X,\Q),
$$
which is a consequence of the fact that $\K_X^\bullet$ is of weight $\le n-1$ as in \cite[Proposition 6.4]{PP25a}.  Thus Lemma \ref{lem: injectivity of (1,1) kernels restricted to weight 2} implies that the second row of \eqref{eqn: last divisor class commuting diagram} is injective.

As a result, we have
\begin{equation}
\label{eqn: Alg and CDiv intersection}
\ker(\tilde c_1^X)=\mathrm{Alg}^1_\Q(X)\cap \CDiv_\Q(X) \,\,\,\, {\rm and} \,\,\,\, \ker(\cl^X_\Q)=\mathrm{Alg}^1_\Q(X).
\end{equation}
Therefore, it suffices to prove the following:

\medskip

\noindent
\textit{Claim.} If $h^1(X)=h^{2n-1}(X)$, then $\ker(\tilde c_1^X)=\ker(\cl^X_\Q)$. If $h^1(X)\neq h^{2n-1}(X)$, then $$\textrm{coker}\left(\ker(\tilde c_1^X)\to \ker(\cl^X_\Q)\right)$$ is an infinite dimensional $\Q$-vector space.

To prove this, consider the natural morphism $H^1(X,\Q)\to \mathit{IH}^1(X,\Q)$. This is injective, since $H^1(X,\Q)$ is pure of weight $1$ for a normal projective variety $X$  
(see e.g. \cite[Example 8.11]{PP25a}) and $\K_X^\bullet$ is of weight $\le n-1$. Accordingly, we have natural inclusions
$$
\mu^*:H^1(X,\Q)\to H^1(\widetilde X,\Q),\quad \mu^*:H^1(X,\O_X)\to H^1(\widetilde X,\O_{\widetilde X})
$$
by Lemma \ref{lem: H1 and H2 of normal variety}.  In particular, since $H^{2n-1} (X, \Q) \simeq  \mathit{IH}^{2n-1}(X,\Q) \simeq \mathit{IH}^1(X,\Q)$,  
we have 
$$h^1(X)=h^{2n-1}(X) \iff H^1(X,\O_X) \simeq H^1(\widetilde X,\O_{\widetilde X}).$$

Recall that we have the identification $\ker(\cl^X_\Q)=\ker(\cl^{\widetilde X}_\Q)$. This allows us to consider instead the dimension of the cokernel of the inclusion
\begin{equation}
\label{eqn: pullback of kernels of class maps}
\mu^*:\ker(\tilde c_1^X)\to \ker(\cl^{\widetilde X}_\Q).    
\end{equation}

\noindent
\textit{Case 1.} $h^1(X)\neq h^{2n-1}(X)$.

In this case, there exists an infinite set $\left\{\alpha_s\right\}_{s\in S}$ of elements $\alpha_s\in H^1(\widetilde X,\O_{\widetilde X})$ such that any nonzero 
$\Q$-linear combination of $\alpha_s$ is not contained in the $\Q$-linear span of $H^1(X,\O_X)$ and $H^1(\widetilde X,\Q)$. Indeed, we can choose infinitely many 
such $\alpha_s$ since 
$$
\dim_\Q H^1(\widetilde X,\O_{\widetilde X})/H^1(X,\O_X)=\infty
$$
and $H^1(\widetilde X,\Q)$ is a finite dimensional $\Q$-vector space. Let $L_s\in \Pic^0(\widetilde X)$ be the line bundle corresponding to the element $\alpha_s$, and $D_s\in \Div(\widetilde X)$ such that $\O_{\widetilde X}(D_s)\isom L_s$. Then, $D_s\in \ker(\cl^{\widetilde X}_\Q)$.

Suppose that  $\sum_{s\in S} a_s D_s\in \mu^*\ker(\tilde c_1^X)$ for some $a_s\in \Q$. Then, for a sufficiently divisible $N\in \mathbb N$,
$$
\widetilde L:=\O_{\widetilde X}\left(N\sum_{s\in S} a_s D_s\right)\in \Pic^0(\widetilde X)
$$
is a line bundle isomorphic to $\mu^* L$, for some $L\in \Pic(X)$ with $c_1(L)=0\in H^2(X,\Z)$. Note that $\widetilde L$ lifts to the class $\sum_{s\in S} (Na_s)\alpha_s\in H^1(\widetilde X,\O_{\widetilde X})$. Say $L$ lifts to a class  $\beta\in H^1(X,\O_X)$. Then, the class
$$
\sum_{s\in S} (Na_s)\alpha_s-\mu^*\beta\in H^1(\widetilde X,\O_{\widetilde X})
$$
gives a trivial line bundle in $H^1(\widetilde X,\O^*_{\widetilde X})$ by the exponential map, hence is contained in $H^1(\widetilde X,\Z)$. This implies that $a_s=0$ for all $s\in S$, by the choice of $\left\{\alpha_s\right\}_{s\in S}$. Therefore, the set $\left\{D_s\right\}_{s\in S}$ is $\Q$-linearly independent in the cokernel of \eqref{eqn: pullback of kernels of class maps}, hence this cokernel is an infinite dimensional $\Q$-vector space.

\smallskip
\noindent
\textit{Case 2.} $h^1(X)= h^{2n-1}(X)$.

In other words, we assume $H^1(X,\O_X)\simeq H^1(\widetilde X,\O_{\widetilde X})$, and would like to show $\ker(\tilde c_1^X)=\ker(\cl^X_\Q)$. Suppose $\widetilde D\in \ker(\cl^{\widetilde X}_\Q)$. Then there exists a sufficiently divisible $N\in \mathbb N$ such that
$$
\O_{\widetilde X}(N\widetilde D)\in \Pic^0(\widetilde X).
$$
Say this lifts to a class $\alpha\in H^1(\widetilde X,\O_{\widetilde X})$. By assumption, there exists an element $\beta\in H^1(X,\O_X)$ such that $\mu^*\beta=\alpha$. Let $L\in H^1(X,\O_X^*)$ be the line bundle associated to $\beta$ by the exponential map. Then $\O_{\widetilde X}(N\widetilde D)=\mu^* L$. By the projection formula, we have
$$
\mu_*\O_{\widetilde X}(N\widetilde D)= L.
$$
On the other hand, we have a natural inclusion of sheaves
$$
\mu_*\O_{\widetilde X}(N\widetilde D)\subset \O_X(\mu_*N\widetilde D).
$$
which is an isomorphism away from a codimension $2$ subvariety of $X$. This implies that $L=\O_X(\mu_*N\widetilde D)$, since they are both reflexive sheaves on $X$. Therefore, $\mu_*N\widetilde D\in \ker(\tilde c_1^X)$ and
$$
N\widetilde D-\mu^*(\mu_*N\widetilde D)
$$
is a $\mathbb Z$-linear combination of exceptional divisors $\sum_{i\in I}a_iE_i$, with $\O_{\widetilde X}(\sum_{i\in I}a_iE_i)\isom\O_{\widetilde X}$. This implies $a_i=0$ for all $i\in I$ by Lemma \ref{lem: linear independence of exceptional divisors}. In conclusion, if $\widetilde D\in \ker(\cl^{\widetilde X}_\Q)$, then $\mu_*\widetilde D\in \ker(\tilde c_1^X)$ and $\widetilde D=\mu^*(\mu_*\widetilde D)$. Equivalently, $\ker(\tilde c_1^X)=\ker(\cl^X_\Q)$.
\end{proof}

In the course of the proof we made use of the following general fact about mixed Hodge structures:

\begin{lem}
\label{lem: injectivity of (1,1) kernels restricted to weight 2}
Let $V_\Q$ be a mixed $\Q$-Hodge structure of weight $\le 2$, with weight filtration $W_\bullet$. Denote $V_\C:=V_\Q\tensor_{\Q}\C$ as the corresponding $\C$-vector space with Hodge filtration $F_\bullet$. If $\gr^F_iV_\C=0$ for $i\notin \left\{0,-1,-2\right\}$, then the natural map
$$
\ker(V_\Q\to \gr^F_0V_\C)\to \ker\big(\gr^W_2V_\Q\to \gr^F_0(\gr^W_2 V_\C)\big)
$$
is injective.
\end{lem}

\begin{proof}
Consider the following commutative diagram, where the bottom short exact sequence is induced by the strictness of the Hodge and weight filtrations (see \cite[Corollary 3.6]{PS08}):
\begin{displaymath}
\xymatrix{
{0}\ar[r]&{W_1V_\Q}\ar[d]\ar[r]&{V_\Q}\ar[d]\ar[r]&{\gr^W_2V_\Q}\ar[d]\ar[r]&{0} \\
{0}\ar[r]&{\gr^F_0(W_1V_\C)}\ar[r]&{\gr^F_0V_\C}\ar[r]&{\gr^F_0(\gr^W_2V_\C)}\ar[r]&{0}
}
\end{displaymath}
The first vertical map is injective, since it is easy to check that $\gr^W_iV_\Q\to\gr^F_0(\gr^W_iV_\C)$ is injective for $i=0,1$. Therefore, we conclude from the Snake Lemma.
\end{proof}

Before proving Theorem \ref{thm:Q-factoriality-defect}, recall that  in Lemma \ref{lem: H1 and H2 of normal variety}  it is shown that $h^{2n-1}(X)= Ih^1(X)$ (resp. $h^{2n-2}(X)= Ih^2(X)$) if $X$ is a normal projective variety (resp. with rational singularities). It is easy to check that $h^1(X)\le Ih^1(X)$ (resp. $h^2(X)\le Ih^2(X)$), using the fact that $\K_X^\bullet$ is of weight $\le n-1$, again by \cite[Proposition 6.4]{PP25a}.

\begin{proof}[Proof of Theorem \ref{thm:Q-factoriality-defect}]
Applying \eqref{eqn: Alg and CDiv intersection} and the Snake Lemma to \eqref{eqn: last divisor class commuting diagram}, we deduce that
$$
\sigma(X)=\dim_\Q (\mathit{IH}^2(X,\Q)\cap \mathit{IH}^{1,1}(X)) -\dim_\Q V_X+\dim_\Q \frac{\mathrm{Alg}^1_\Q(X)}{\mathrm{Alg}^1_\Q(X)\cap \CDiv_\Q(X)}.
$$

Therefore, the main statements in Theorem \ref{thm:Q-factoriality-defect} are immediate from Proposition \ref{prop: algebraic equivalence and Cartier divisors}. Furthermore, if $X$ has rational singularities, then $\O_X\isom\DB_X^0\isom I\DB_X^0$. In particular,
$$
\gr^F_0H^2(X,\C)=\gr^F_0\mathit{IH}^2(X,\C)=H^2(X,\O_X).
$$
Using the fact that $\K_X^\bullet$ is of weight $\le n-1$, the kernel $K$ of the natural morphism
$$
H^2(X,\Q)\to \mathit{IH}^2(X,\Q)
$$
is a mixed $\Q$-Hodge structure of weight $\le 1$, with $\gr^F_pK_\C=0$ for all $p\ge0$. By \cite[Proposition 5.2]{PP25a}, 
$\gr^F_pK_\C=0$ for all $p\in \Z$, so $K=0$.

All of this implies that $H^2(X,\Q)\to \mathit{IH}^2(X,\Q)$ is an injection of (polarizable) pure $\Q$-Hodge structures of weight $2$, with cokernel only consisting of the $(1,1)$-component. From the fact that the category of polarizable Hodge structures is semisimple (see e.g. \cite[Corollary 2.12]{PS08}), we conclude that
$$
\sigma(X)= Ih^2(X)-h^2(X)=h^{2n-2}(X)-h^2(X),
$$
where the last equality follows from Lemma  \ref{lem: H1 and H2 of normal variety}.
\end{proof}

\begin{rmk}\label{rmk:formula-sigma-DB}
Note that by \cite[Theorem A]{PP25a}, which gives the symmetry of the boundary of the Hodge-Du Bois diamond for varieties with rational singularities, 
the last statement of Theorem \ref{thm:Q-factoriality-defect} may also be written as
$$\sigma (X) = \h^{n-1, n-1} (X) - \h^{1,1} (X).$$
\end{rmk}

\section{Local analytic versions of $\Q$-factoriality and the Lefschetz (1,1)-theorem}

\subsection{Local analytic Lefschetz (1,1)-theorem} We first establish a local analogue of the cycle class map for a germ of a normal analytic variety. Recall from the discussion following Theorem \ref{thm:Lefschetz-(1,1)} that the cycle class map
$$
\cl_d\tensor \Q: Z_d(X)\tensor \Q\to H^{BM}_{2d}(X,\Q),
$$
is well defined for an analytic variety $X$. Sheaf theoretically, for an irreducible subvariety $Z\subset X$ of dimension $d$, its class $\cl_d(Z)$ is the cohomology class induced by the composition of the following morphisms
\begin{equation}
\label{eqn:Borel Moore composition}
\Q_X\to \Q_Z\to \dual \Q_Z[-2d]\to \dual \Q_X[-2d]    
\end{equation}
where the first and third maps are adjunction of the inclusion $\iota_Z:Z\to X$ and its dual, and the second map is the composition
$$
\Q_Z\to \IC_Z[-d]\isom \dual\IC_Z[-d]\to \dual \Q_Z[-2d].
$$
See also the discussion at the beginning of Section \ref{scn:(1,1)}.

Following the notation in \cite[Section 1]{Kawamata88}, for a normal analytic variety $X$ of dimension $n$, we denote by $(X,x)$ the germ at a point $x\in X$ and
$$
\Div_\Q^{\rm an}(X,x):=\varinjlim_{x\in U}\Div_\Q^{\rm an}(U),\quad \CDiv_\Q^{\rm an}(X,x):=\varinjlim_{x\in U}\CDiv_\Q^{\rm an}(U)
$$
where the direct limit is taken over all the analytic open neighborhoods of $x$ under inclusion. Then, by definition, the \emph{defect of local analytic $\Q$-factoriality} of $X$ at $x$ is
$$
\sigma^{\rm an}(X;x)=\dim_\Q \Div_\Q^{\rm an}(X,x)/\CDiv_\Q^{\rm an}(X,x).
$$
Taking the direct limit of the cycle class maps
$$
\cl_{n-1}\tensor \Q:\Div_\Q^{\rm an}(U)\to H^{BM}_{2n-2}(U,\Q)
$$
over all $U$ as above, we obtain the local cycle class morphism
\begin{equation}
\label{eqn:naive local cycle class map}
\cl_\Q^{\rm an}(X,x):\Div_\Q^{\rm an}(X,x)\to \H^{-2n+2}(\dual \Q_X)_x
\end{equation}
where the target is the stalk at $x$ of the $\Q$-constructible sheaf $\H^{-2n+2}(\dual \Q_X)$. Alternatively, since cohomology commutes with taking stalks, the class $\cl_\Q^{\rm an}(D)$ of an irreducible divisor $D\in \Div_\Q^{\rm an}(X,x)$ is obtained by taking the stalk at $x$ of \eqref{eqn:Borel Moore composition}:
$$
\Q_x\to \iota_x^*\dual\Q_X[-2n+2],
$$
where $\iota_x:x\hookrightarrow X$ is the inclusion. This sheaf-theoretic interpretation allows us to further restrict the image of the local cycle class morphism using Saito's theory \cite{Saito90}, in the form of a local analytic Lefschetz (1,1)-theorem that we present next.

\begin{thm}\label{thm:local Lefschetz-(1,1)}
Let $(X, x)$ be an analytic germ of a normal algebraic variety of dimension $n$, and let $\mu \colon \widetilde X\to X$ be a resolution of singularities with $\mu$ projective. Then: 

\noindent
(i)  There exists a local cycle class morphism
$$
{\rm cl}^{\rm an}_\Q(X,x) \colon \Div^{\rm an}_\Q(X,x)\to \ker\left(\H^{-n+2}(\IC_X)_x\to (R^2\mu_*\O_{\widetilde X})_x\right)
$$
which is surjective. 

\noindent
(ii) If $R^1\mu_*\O_{\widetilde X}=0$, then the kernel of the morphism in (i) is $\CDiv^{\rm an}_\Q(X,x)$.
\end{thm}

\begin{rmk}
For technical reasons, we require the germ $x\in X$ to arise from an algebraic variety. For example, we need to work with the adjunction map
$$
\Q_X^H\to \mu_*\Q_{\widetilde X}^H 
$$
in the bounded derived category of mixed Hodge modules $D^b{\rm MHM}(X)$. However, the pushforward functor is not generally defined in the derived category for morphisms of analytic varieties, and the existence of such an adjunction is also not known. Additionally, we use the description of the Du Bois complexes of $X$ as the graded de Rham complexes of the trivial Hodge module $\Q^H_X$, as established in \cite{Saito00} in the context of algebraic varieties.
\end{rmk}

%Recall that we are using the notation and background introduced in \cite{PP25a}, especially Section 5.
To set things up, note that there is an isomorphism
$$
\gr^F_0\DR(\IC_X^H)\isom R\mu_*\O_{\widetilde X}[n],
$$
and moreover $\gr^F_{i}\DR(\IC^H_X)=0$ for all $i>0$; see, for example, \cite[Lemma 4.2]{Park23}. Therefore there exists a natural morphism
$$
\IC_X\to \gr^F_{0}\DR(\IC^H_X)
$$
in the derived category of sheaves on $X$. Taking the $(-n+2)$-nd cohomology and its stalk at $x$, we obtain a morphism
$$
\H^{-n+2}(\IC_X)_x\to \H^{-n+2}(\gr^F_0\DR(\IC_X^H))_x,
$$
whose kernel coincides with the target of the cycle class morphism in (i).

We first prove the existence of the cycle class morphism in Theorem \ref{thm:local Lefschetz-(1,1)} (i). We start by showing that the composed map \eqref{eqn:Borel Moore composition} factors through the intersection complex.

\begin{lem}
\label{lem:Borel Moore composition factors thru IC}
In the setting of Theorem \ref{thm:local Lefschetz-(1,1)}, for an irreducible divisor $D\subset X$, the composition in \eqref{eqn:Borel Moore composition} coincides with the composition of the maps of underlying $\Q$-complexes of objects in $D^b{\rm MHM}(X)$:
$$
\Q_X^H\to \Q_D^H\to \dual \Q_D^H(-n+1)[-2n+2]\to \dual \Q_X^H(-n+1)[-2n+2].
$$
Furthermore, the resulting composition $\Q_X^H\to\dual \Q_X^H(-n+1)[-2n+2]$ (non-canonically) factors through the natural morphism
$$
\IC_X^H(1)[-n+2]\to \dual \Q_X^H(-n+1)[-2n+2]
$$
in $D^b{\rm MHM}(X)$.
\end{lem}

\begin{proof}
Note that the trivial Hodge module $\Q^H_Z$ is not generally defined for arbitrary varieties $Z$. For a subvariety $Z$ of a germ $(X,x)$ with inclusions $\iota_Z:Z\to X$ and $j_Z:X\sm Z\to X$, the ideal sheaf of $Z$ in $X$ is globally generated by holomorphic functions of $X$, and thus we can define functors $(j_Z)_!j_Z^{-1}$, $(j_Z)_*j_Z^{-1}$, $(\iota_Z)_*\iota_Z^*$, and $(\iota_Z)_*\iota_Z^!$ in $D^b{\rm MHM}(X)$ with the adjunction maps as described in \cite[(2.30.1)]{Saito90}. These functors are compatible with the corresponding functors for the underlying $\Q$-complexes. In particular, the trivial Hodge module $\Q^H_Z$ is $(\iota_Z)_*\iota_Z^*\Q^H_X$.

As a consequence, the first map of \eqref{eqn:Borel Moore composition} is the map of underlying $\Q$-complexes of associated objects in $D^b{\rm MHM}(X)$, and the third map is its dual. Hence, it suffices to interpret the second map as the map of underlying $\Q$-complexes of objects in $D^b{\rm MHM}(X)$.

For an irreducible subvariety $Z\subset X$ of dimension $d$, there exists a natural morphism
$$
\Q^H_Z[d]\to \IC_Z^H
$$
in $D^b{\rm MHM}(X)$, which is obtained from the isomorphism $\gr^W_d\H^d(\Q^H_Z)\isom \IC^H_Z$. Indeed, this follows from the mixed Hodge module formalism and the argument in \cite[Section 4.5]{Saito90}. Note that the original argument in \cite{Saito90} is for algebraic varieties, but the aforementioned descriptions of pullback and pushforward functors for the inclusions $\iota_Z$ and $j_Z$ are sufficient to run the same argument for analytic germs. Hence, the second map of \eqref{eqn:Borel Moore composition} can be realized as the composition of the maps of the underlying $\Q$-complexes of
$$
\Q_Z^H\to \IC_Z^H[-d]\isom \dual\IC_Z^H(-d)[-d]\to \dual \Q_Z^H(-d)[-2d],
$$
where the middle isomorphism is induced by the canonical polarization of $\IC_Z^H$. Therefore, this completes the proof of the first statement, when $Z=D$ is an irreducible divisor.

Next, we prove that a morphism $\Q_X^H\to\dual \Q_X^H(-n+1)[-2n+2]$ in $D^b{\rm MHM}(X)$ necessarily factors through $\IC_X^H(1)[-n+2]$. Recall from \cite[Proposition 6.4]{PP25a} that the RHM-defect object $\K^\bullet_X$ is of weight $\le n-1$. Consider the distinguished triangle
$$
\dual\IC^H_X(-n+1)[-n+2]\to \dual \Q_X^H(-n+1)[-2n+2]\to \dual\K^\bullet_X(-n+1)[-n+2] \xrightarrow{+1}
$$
which is dual to the triangle ($\ref{eqn:RHM-object}$) up to a shift and a Tate twist. The first term is isomorphic to $\IC_X^H(1)[-n+2]$ by the canonical polarization of $\IC^H_X$, and the third term is of weight $\ge 1$ as an object in $D^b{\rm MHM}(X)$. Recall that $\Q^H_X$ is of weight $\le 0$. From the mixed Hodge module formalism (see, for example, \cite[Lemma 14.3]{PS08}), we have
$$
\mathrm{Hom}_{D^b {\rm MHM}(X)}\left(\Q^H_X,\dual\K^\bullet_X(-n+1)[-n+2]\right)=0,
$$
and this implies the last statement of the Lemma.
\end{proof}

The Lemma  above shows that the image of the local cycle class morphism \eqref{eqn:naive local cycle class map} factors through the image of the natural morphism
\begin{equation}
\label{eqn: injection of stalks}
\H^{-n+2}(\IC_X)_x\to \H^{-2n+2}(\dual\Q_X)_x. 
\end{equation}
This map of stalks is an injection:

\begin{lem}
\label{lem:image of local cycle class map}
In the setting of Theorem \ref{thm:local Lefschetz-(1,1)}, the map \eqref{eqn: injection of stalks} is injective, and the image of the local cycle class morphism \eqref{eqn:naive local cycle class map} is contained in
$$
\ker\left(\H^{-n+2}(\IC_X)_x\to (R^2\mu_*\O_{\widetilde X})_x\right).
$$
\end{lem}

\begin{proof}
Consider the distinguished triangle
$$
\IC_X\to\dual \Q_X[-n]\to \dual \K_X^\bullet\xrightarrow{+1}.
$$
Here, we only consider the underlying $\Q$-structure of the RHM-defect object $\K_X^\bullet$. Recall that $\dual\K_X^\bullet\in{}^p\!D^{\ge 0}_c(X,\Q)$ and $\dual\K_X^\bullet$ is supported  on a closed subset of $X$ of  codimension at least  $2$. This implies $\H^{-n+1}(\dual\K_X^\bullet)=0$. Therefore, from the long exact sequence of constructible cohomologies, \eqref{eqn: injection of stalks} is injective.

Combined with Lemma \ref{lem:Borel Moore composition factors thru IC}, the image of the local cycle class morphism \eqref{eqn:naive local cycle class map} for divisors is contained in 
$$
\H^{-n+2}(\IC_X)_x\subset \H^{-2n+2}(\dual\Q_X)_x.
$$
Note that we have a natural isomorphism of graded de Rham complexes
$$
\gr^F_{-n}\DR(\Q_X^H[n])\isom\gr^F_{-n}\DR(\IC_X^H) \simeq \mu_*\omega_{\widetilde X} \simeq R \mu_*\omega_{\widetilde X} 
$$
by \cite[Corollary 0.3]{Saito00}, where the last isomorphism is due to the Grauert-Riemenschneider vanishing theorem. Taking its dual and a Tate twist, we have a natural isomorphism
$$
\gr^F_{1}\DR\left(\IC_X^H(1)[-n+2]\right)\isom \gr^F_{1}\DR\left(\dual \Q_X^H(-n+1)[-2n+2]\right)\isom R\mu_*\O_{\widetilde X}[2].
$$
Since $\gr^F_1\DR(\Q^H_X)=0$, we have the following commutative diagram induced by Lemma \ref{lem:Borel Moore composition factors thru IC}:
\begin{displaymath}
\xymatrix{
{\Q_X}\ar[d]\ar[r]&{\IC_X[-n+2]}\ar[d]\ar[r]&{\dual \Q_X(-n+1)[-2n+2]}\ar[d] \\
{0}\ar[r]&{R\mu_*\O_{\widetilde X}[2]}\ar[r]^{\sim}&{R\mu_*\O_{\widetilde X}[2]}
}
\end{displaymath}
where the top row is the map of underlying $\Q$-structures and the bottom row is the map of the associated graded terms 
$\gr^F_1\DR(\cdot)$ of de Rham complexes. Note that the vertical maps commute with the rows, because $\gr^F_{>1}\DR(\cdot)=0$. Consequently, taking the stalk at $x$, we conclude that the image of the local cycle class morphism is contained in
$$
\ker\left(\H^{-n+2}(\IC_X)_x\to (R^2\mu_*\O_{\widetilde X})_x\right).
$$
\end{proof}

This establishes the cycle class morphism in Theorem \ref{thm:local Lefschetz-(1,1)} (i). To prove its surjectivity, we need a local analytic analogue of Lemma \ref{lem: difference of h2 and IH2}.

\begin{lem} \label{lem:local difference of H2 and IH2}
In the setting of Theorem \ref{thm:local Lefschetz-(1,1)}, denote by $\left\{E_i\right\}_{i\in I}$ the set of $\mu$-exceptional divisors in $\widetilde X$. Then we have
$$
\dim_\Q (R^2\mu_*\Q_{\widetilde X})_x-\dim_\Q\H^{-n+2}(\IC_X)_x=|I|.
$$
\end{lem}

Recall that $x\in X$ is an analytic germ, and a $\mu$-exceptional divisor $E$ satisfies $x\in\mu(E)$.

\begin{proof}
As in the proof of Lemma \ref{lem: difference of h2 and IH2}, we have the decomposition
$$
R\mu_*\Q_{\widetilde X}[n] \simeq \IC_X[0]\oplus \M^\bullet.
$$
where $\M^\bullet$ is supported inside the locus $Z\subset X$ where $\mu$ is not an isomorphism. It suffices to show that
$$
\dim_\Q \H^{-n+2}(\M^\bullet)_x=|I|.
$$

Denote by $\iota:Z\hookrightarrow X$ the closed embedding. Note that $\iota^!\IC_X\in {}^p\!D^{\ge 1}_c(Z,\Q)$ (see \cite[Proposition 8.2.5]{HTT}). Since $\dim Z\le n-2$, we have
$$
\H^{-n+2}(\iota^!\IC_X[0])_x=0.
$$
From $(\iota^!\M^\bullet)_x=(\M^\bullet)_x$, it suffices to show that
$$
\dim_\Q \H^{2}(\iota^!R\mu_*\Q_{\widetilde X})_x=|I|
$$
or equivalently,
$$
\dim_\Q R^2\mu_*(\tilde\iota^!\Q_{\widetilde X})_x=|I|
$$
where $\tilde\iota:\mu^{-1}(Z)\hookrightarrow \widetilde X$. Then we have $\tilde\iota^!\Q_{\widetilde X}=\dual \Q_{\mu^{-1}(Z)}[-2n]$ and
$$
R^2\mu_*(\tilde\iota^!\Q_{\widetilde X})_x=\varinjlim_{x\in U}H^{BM}_{2n-2}(\mu^{-1}(U\cap Z),\Q).
$$
by the proper base change theorem. The top Borel-Moore homology $H^{BM}_{2n-2}(\mu^{-1}(U\cap Z),\Q)$ is a $\Q$-vector space with a basis consisting of the classes of $(n-1)$-dimensional irreducible components of $\mu^{-1}(U\cap Z)$. However the set such irreducible components coincides with the set $\left\{E_i\right\}_{i\in I}$ of exceptional divisors, for small enough neighborhoods $U$ of $x$. This completes the proof.
\end{proof}

\begin{proof}[Proof of Theorem \ref{thm:local Lefschetz-(1,1)}]
For a neighborhood $U$ of $x$, we have the adjunction map
$$
\mu_*\dual \Q^H_{\widetilde U}\to \dual \Q_U^H 
$$
in $D^b{\rm MHM}(U)$, where $\widetilde U:=\mu^{-1}(U)$. This induces the following commutative diagram
\begin{displaymath}
\xymatrix{
{R\mu_*\dual\Q_{\widetilde U}}\ar[d]\ar[r]&{R\mu_*\O_{\widetilde U}[2n]}\ar[d]^{\simeq} \\
{\dual \Q_U}\ar[r]&{R\mu_*\O_{\widetilde U}[2n]}
}
\end{displaymath}
where the right column is obtained by taking the associated graded terms of the de Rham complexes, $\gr^F_n\DR(\cdot)$. Taking the hypercohomology groups  $\HH^{-2n+2}(\cdot)$, we have the following commutative diagram
\begin{displaymath}
\xymatrix{
{\Div^{\rm an}_\Q(\widetilde U)}\ar[d]_{\mu_*}\ar[r]^{{\rm cl}_\Q^{\rm an}}&{H^{BM}_{2n-2}(\widetilde U,\Q)}\ar[d]_{\mu_*}\ar[r]&{H^2(\widetilde U,\O_{\widetilde U})}\ar[d]^{\simeq} \\
{\Div^{\rm an}_\Q(U)}\ar[r]^{{\rm cl}_\Q^{\rm an}}&{H^{BM}_{2n-2}(U,\Q)}\ar[r]&{H^2(\widetilde U,\O_{\widetilde U})}
}
\end{displaymath}
Since $\mu$ is projective, for a small Stein neighborhood $U$ of $x$, any line bundle of $\widetilde U$ is induced by the divisor $\widetilde D\in \Div^{\rm an}_\Q(\widetilde U)$ (see \cite{GR58}). Therefore, the first row is exact, due to the exponential exact sequence. Denote
$$
\Div_\Q^{\rm an}(\widetilde X,\mu^{-1}(x)):=\varinjlim_{x\in U}\Div_\Q^{\rm an}(\widetilde U),
$$
and take the direct limit of the commutative diagram above over the open neighborhoods $x\in U$ under inclusion:
\begin{displaymath}
\xymatrix{
{\Div^{\rm an}_\Q(\widetilde X,\mu^{-1}(x))}\ar[d]_{\mu_*}\ar[r]^{{\rm cl}_\Q^{\rm an}}&{(R^{-2n+2}\mu_*\dual\Q_{\widetilde X})_x}\ar[d]_{\mu_*}\ar[r]&{(R^2\mu_*\O_{\widetilde X})_x}\ar[d]^{\sim} \\
{\Div^{\rm an}_\Q(X,x)}\ar[r]^{{\rm cl}_\Q^{\rm an}}&{\H^{-2n+2}(\dual\Q_X)_x}\ar[r]&{(R^2\mu_*\O_{\widetilde X})_x}
}
\end{displaymath}
Since the first row is exact, we obtain the commutative diagram
\begin{equation}\label{eqn:local class map resolution diagram}
\xymatrix{
{\Div^{\rm an}_\Q(\widetilde X,\mu^{-1}(x))}\ar[d]_{\mu_*}\ar@{->>}[r]^-{{\rm cl}_\Q^{\rm an}}&{\ker\left((R^{-2n+2}\mu_*\dual\Q_{\widetilde X})_x\to(R^2\mu_*\O_{\widetilde X})_x\right)}\ar[d]_{\mu_*} \\
{\Div^{\rm an}_\Q(X,x)}\ar[r]^-{{\rm cl}_\Q^{\rm an}}&{\ker\left(\H^{-n+2}(\IC_X)_x\to(R^2\mu_*\O_{\widetilde X})_x\right)}
}
\end{equation}
using Lemma \ref{lem:image of local cycle class map}. As in its proof, $\mu_*\dual\Q^H_{\widetilde X}$ is of weight $0$ and $\dual \K_X^\bullet[n]$ is of weight $\ge 1$, which implies that the adjunction map $\mu_*\dual \Q^H_{\widetilde X}\to \dual \Q_X^H$ factors non-canonically through $\IC_X^H[n]$. Together with the Decomposition Theorem, this implies that the second column is surjective. 
Therefore ${\rm cl}^{\rm an}_\Q(X,x)$ is surjective as well, as desired.

Furthermore, the kernel of the second column is a finite dimensional $\Q$-vector space with basis consisting of the classes of $\mu$-exceptional divisors, by Lemmas \ref{lem: linear independence of exceptional divisors} and \ref{lem:local difference of H2 and IH2}. Therefore, for $D\in \Div^{\rm an}_\Q(X,x)$ with ${\rm cl}_\Q^{\rm an}(X,x)(D)=0$, there exists a $\Q$-linear combination $E$ of exceptional divisors such that the divisor class of $\mu^{-1}_*(D)+E$ is zero, that is
$$
{\rm cl}_\Q^{\rm an}(\widetilde X,\mu^{-1}(x))(\mu^{-1}_*(D)+E)=0
$$
where ${\rm cl}_\Q^{\rm an}(\widetilde X,\mu^{-1}(x))$ denotes the class map of the first row in \eqref{eqn:local class map resolution diagram}, and $\mu^{-1}_*(D)$ is the strict transform of $D$.

Suppose $R^1\mu_*\O_{\widetilde X}=0$. Since $\mu^{-1}_*(D)+E$ is numerically relatively trivial over a small neighborhood of $x$, there exists a sufficiently divisible $N\in \mathbb N$ such that $ND$ is a Cartier divisor, by \cite[Proposition 12.1.4]{KM92}. Consequently, the kernel of ${\rm cl}_\Q^{\rm an}(X,x)$ is $\CDiv^{\rm an}_\Q(X,x)$.
\end{proof}

\subsection{The local analytic $\Q$-factoriality defect}\label{scn:loc-an-def}
We next come to the main point of this chapter, the description of the local analytic $\Q$-factoriality defect in Hodge-theoretic terms. 

\begin{proof}[Prood of Theorem \ref{thm:local analytic Q-factoriality defect}]

By Theorem \ref{thm:local Lefschetz-(1,1)}, the only statement that remains to be proven is $(i)\Rightarrow(ii)$. We argue by contradiction; suppose 
$(R^1\mu_*\O_{\widetilde X})_x\neq0$ for some resolution $\mu$. Then there exists an infinite set $\left\{\alpha_s\right\}_{s\in S}$ of elements $\alpha_s\in (R^1\mu_*\O_{\widetilde X})_x$ such that any nonzero $\Q$-linear combination of $\alpha_s$ is not contained in $(R^1\mu_*\Q_{\widetilde X})_x\subset (R^1\mu_*\O_{\widetilde X})_x$. Using the exponential exact sequence on $\widetilde X$, the proof of Proposition \ref{prop: algebraic equivalence and Cartier divisors} Case 1 applies in a similar fashion.

In brief, let $L_s$ be the line bundle associated to $\alpha_s$, and let $D_s$ be a divisor satisfying $L_s\isom \O_{\widetilde X}(D_s)$. Then for a nonzero $\Q$-linear combination $D$ of $D_s$, the pushforward $\mu_*(D)$ is never $\Q$-Cartier. Indeed, if it were, then $D-\mu^*\mu_*(D)$ would be numerically equivalent to zero relative to $X$. By Lemma \ref{lem: linear independence of exceptional divisors}, we have $D=\mu^*\mu_*(D)$ which contradicts the fact that any nonzero $\Q$-linear combination of $\alpha_s$ is not contained in $(R^1\mu_*\Q_{\widetilde X})_x$. Therefore, any $\Q$-linear combination of $\mu_*(D_s)$ is not contained in $\CDiv^{\rm an}_\Q(X,x)$, which implies that $\sigma^{\rm an}(X;x)$ is infinite.
\end{proof}

We derive some interesting consequences of this result.

In \cite[Lemma 1.1]{Kawamata88}, Kawamata showed that for a normal algebraic variety $X$ with $R^1\mu_*\O_{\widetilde X}=R^2\mu_*\O_{\widetilde X}=0$, the defect $\sigma(X)$ is finite. (This can also be obtained from \cite[Satz 6.1]{Flenner81}.) As an immediate corollary, we recover this result and deduce a stronger statement:

\begin{cor}
Let $X$ be a normal algebraic variety with $R^1\mu_*\O_{\widetilde X}=0$ for a resolution $\mu\colon \widetilde X\to X$. Then the defect of $\Q$-factoriality $\sigma(X)$ is finite.

Furthermore, if $\H^{-n+2}(\IC_X)=0$, then $X$ is locally analytically $\Q$-factorial.
\end{cor}

The next consequence constrains the complement of the locus where the local analytic $\Q$-factoriality defect is finite, under an additional depth assumption.

\begin{cor}\label{cor:infinite-defect}
Let $X$ be a normal algebraic variety.  If $X$ is $S_3$, then the set of points $x$ where $\sigma^{\rm an} (X; x)$ is infinite is either empty or a Zariski closed subset of pure codimension $2$.
\end{cor}
\begin{proof} 
By Theorem \ref{thm:local analytic Q-factoriality defect}, we are equivalently looking at the locus where $(R^1\mu_*\O_{\widetilde X})_{x}\neq 0$, for some resolution $\mu:\widetilde X\to X$. Since for each $i \ge 1$ the codimension of the support of $R^i \mu_* \O_{\widetilde X}$ is at least $i +1$, this locus is in any case of codimension at least $2$; moreover, if it were of codimension at least $3$, then $X$ would have rational singularities in codimension $2$. By the Lemma below, due to Koll\'ar, in combination with the $S_3$ condition this would imply that $R^1\mu_*\O_{\widetilde X} = 0$ everywhere.
\end{proof}

The following result and proof are due to J. Koll\'ar \cite{Kollar}. We thank him for allowing us to include them here. 

\begin{lem}\label{lem:Kollar}
Let $X$ be a normal variety over a field of characteristic zero. Then the following are equivalent:

\noindent
(i)~ $R^1\mu_*\O_{\widetilde X} = 0$ for some (equivalently, any) resolution $\mu:\widetilde X\to X$.

\noindent
(ii)~$X$ is $S_3$ and has rational singularities in codimension $2$.
\end{lem}
\begin{proof}
Let $x$ be any point of $X$, a priori not necessarily closed. In order to check the statement at $x$, it is enough to localize, and hence assume that 
$x$ is a closed point on $X$ of dimension $d$. We consider a resolution $\mu \colon \widetilde X \to X$, and denote $E = \mu^{-1} (x)_{\rm red}$.
The statement is clear for $d = 2$, so we may assume $d \ge 3$.

The Leray spectral sequence computing $H^\bullet_E (\widetilde X, \O_{\widetilde X})$  yields an exact sequence 
$$H^1_E (\widetilde X, \O_{\widetilde X}) \to H^0_x (X, R^1 \mu_* \O_{\widetilde X}) \overset{\alpha}{\longrightarrow} H^2_x(X, \O_X)  \to H^2_E (\widetilde X, \O_{\widetilde X})$$
Moreover, by formal duality, $H^i_E (\widetilde X, \O_{\widetilde X})$ is dual to $R^{d - i} \mu_* \omega_{\widetilde X}$ up to completion, and the latter is $0$ for $i < d$ by the Grauert-Riemenschneider theorem. This means that the map $\alpha$ is an isomorphism, which establishes an equivalence 
between  $R^1\mu_*\O_{\widetilde X} = 0$ and the $S_3$ property.
\end{proof}

\section{$\Q$-factoriality and the Hodge-Du Bois numbers of threefolds}

\subsection{Formula for the (local analytic) $\Q$-factoriality defect}
One of our main applications of the general Theorem \ref{thm:Q-factoriality-defect} is the analysis of the relationship between $\Q$-factoriality and the middle rhombus in the Hodge-Du Bois diamond of a projective threefold. This is described in Theorem \ref{thm:analytic-Q-factoriality-threefolds}, and this section is devoted to the proof of that theorem. Besides Theorem \ref{thm:Q-factoriality-defect}, the proof relies heavily on 
results and techniques from the companion paper \cite{PP25a}.

\begin{proof}[{Proof of Theorem \ref{thm:analytic-Q-factoriality-threefolds}}]
By \cite[Theorem A]{PP25a}, we have 
$$
\underline h^{0,q}(X)=\underline h^{q,0}(X)=\underline h^{3,3-q}(X)=\underline h^{3-q,3}(X)
$$
for all $0\le q\le 3$;  in \emph{loc. cit.} it is actually shown that they are equal to the respective intersection Hodge numbers $I\h^{p,q}(X)$. By Theorem \ref{thm:Q-factoriality-defect}, this implies
$$
\sigma(X)=h^4 (X)- h^2 (X)=\underline h^{2,2}(X)-\underline h^{1,1}(X).
$$
In particular, $\underline h^{2,2}(X)\ge\underline h^{1,1}(X)$. Additionally, due to the rational singularities hypothesis, by \cite[Proposition 7.4]{PP25a} we have $\DB_X^2=I\DB_X^2$, hence
$$
\underline h^{2,q}(X)=I\underline h^{2,q}(X)
$$
for all $0\le q\le 3$.

Now since $X$ is a threefold with rational singularities,  we have $\lcdef (X) = 0$ by \cite{DT16}. Moreover, $X$ is a rational homology manifold away from a finite set, by \cite[Corollary 7.5]{PP25a}. This means that the RHM-defect object $\K^\bullet_X$ is a single mixed Hodge module 
$\mathcal V^H$, supported on this finite set, and the triangle ($\ref{eqn:RHM-object}$) becomes
\begin{equation}\label{eqn:special-triangle}
\mathcal V^H \to  \Q_X^H[n] \to  \IC_X^H\xrightarrow{+1}.    
\end{equation}
A mixed Hodge module supported on a point is a mixed Hodge structure, which we denote by $V^H_x$ for each point $x$ in the support of 
$\mathcal V^H$.
Denoting by $a_X:X\to {\rm  pt}$ the constant map to a point, we have 
$$
V^H[0]={a_X}_*\mathcal V^H,
$$
where $V^H$ is a mixed $\Q$-Hodge structure of weight $\le 2$ (by \cite[Proposition 6.4]{PP25a}), isomorphic to the direct sum of mixed Hodge structures $V^H_x$.\footnote{In fact it is not hard to check that $V^H$ is pure of weight $2$.} 

Passing to the long exact sequence of cohomology associated to the exact triangle ($\ref{eqn:special-triangle}$), we obtain an exact sequence of mixed $\Q$-Hodge structures
\begin{equation}
\label{eqn: LES of MHS of threefold}
0\to H^2(X,\Q)\to \mathit{IH}^2(X,\Q)\to V^H\to H^3(X,\Q)\to \mathit{IH}^3(X,\Q)\to 0, 
\end{equation}
while for $i \neq 2, 3$ we have
$$H^i(X,\Q)\simeq \mathit{IH}^i (X,\Q).$$

One immediate consequence is that there is a surjection
$$
H^{1,2}(X)\twoheadrightarrow \mathit{IH}^{1,2}(X),
$$
hence $\underline h^{1,2}(X)\ge I\underline h^{1,2}(X)=I\underline h^{2,1}(X)=\underline h^{2,1}(X)$. This proves the first part of the theorem.\footnote{The inequality 
$\h^{1,2}(X)\ge \h^{2,1}(X)$ follows in fact directly from \cite[Lemma 3.23]{FL24a}, which gives $\h^{0,3} (X)   + \h^{1,2} (X) \ge \h^{3,0} (X) + \h^{2,1} (X)$, combined with the equality $\h^{0,3} (X) = \h^{3,0} (X)$, which holds since $X$ has rational singularities.}

\smallskip

Next we focus on the final assertion, regarding local analytic $\Q$-factoriality.
For a closed point $\iota_x:x\hookrightarrow X$, we pull back the distinguished triangle ($\ref{eqn:special-triangle}$) to get
\begin{equation}
\label{eqn: pullback to point of Q to IC}
\iota_x^*\mathcal V^H\to \Q^H[3]\to \iota_x^*\IC_X^H\xrightarrow{+1}.
\end{equation}
Since $\mathcal V^H$ is supported at a finite set of closed points $\left\{x_s\right\}_{s\in S}$, the pullback 
$$V^H_x = \iota_x^*\mathcal V^H$$ 
is a $\Q$-Hodge structure supported in cohomological degree $0$. 
Furthermore, from the long exact sequence of cohomology of \eqref{eqn: pullback to point of Q to IC}, we have
$$
\H^{-1}(\iota_x^*\IC^H_X)\isom V^H_x.
$$
Making use of the last statement in Theorem \ref{thm:local analytic Q-factoriality defect}, altogether we deduce that
$$
\sigma^{\rm an}(X;x)\isom \dim_\Q V^H_x.
$$
Therefore, $\sigma^{\rm an}(X;x)=0$  away from the finite set $\left\{x_s\right\}_{s\in S}$, and
\begin{equation}\label{eqn:an-def}
\sum_{s\in S}\sigma^{\rm an}(X;x_s)=\sum_{s\in S}\dim_\Q V^H_{x_s}=\dim_\Q V^H.
\end{equation}
Now ($\ref{eqn: LES of MHS of threefold}$) and the identities of Hodge numbers listed earlier in the proof imply that we have
$$
\dim_\Q V^H=I\underline h^{1,1}(X)-\underline h^{1,1}(X)+\underline h^{1,2}(X)-I\underline h^{1,2}(X),
$$
which in turn implies that
\begin{equation}
\label{eqn: dim of V^H=K_X}
\dim_\Q V^H=\underline h^{2,2}(X)-\underline h^{1,1}(X)+\underline h^{1,2}(X)-\underline h^{2,1}(X).
\end{equation}
The second equality follows from $I\underline h^{1,2}(X)=\underline h^{2,1}(X)$ (explained earlier in the proof) and $I\underline h^{1,1}(X)= I\underline h^{2,2}(X)=\underline h^{2,2}(X)$,  which is due to the isomorphism $H^4 (X, \Q) \simeq {\it IH}^4 (X, \Q)$ of Hodge structures.

Since $\sigma(X)=\underline h^{2,2}(X)-\underline h^{1,1}(X)$, the identities \eqref{eqn:an-def} and \eqref{eqn: dim of V^H=K_X} complete the proof.
\end{proof}

We can now easily obtain the characterization of projective threefolds with rational singularities that are rational homology manifolds.

\begin{proof}[{Proof of Corollary \ref{cor:RHM-threefolds}}]
The equivalence between (i) and (ii) is in fact a special case of \cite[Theorem 10.1]{PP25a} since, as explained in the proof of Theorem 
\ref{thm:analytic-Q-factoriality-threefolds}, $X$ is a rational homology manifold away from a finite set of points.

To show that (ii) is equivalent to (iii), note first that  by  \cite[Theorem A]{PP25a} we have in any case the symmetry of the boundary of the 
Hodge-Du Bois diamond. If we have full symmetry, then by Theorem \ref{thm:analytic-Q-factoriality-threefolds} it is obvious that $\sigma^{\rm an}(X;x)= 0$ for all $x \in X$, that is $X$ is locally analytically $\Q$-factorial. 
Conversely, if all $\sigma^{\rm an}(X;x)= 0$, then the same theorem tells us that 
$$  - \underline h^{2,2}(X) + \underline h^{1,1}(X) = \underline h^{1,2}(X)-\underline h^{2,1}(X),$$
and moreover that the quantity on the left is nonpositive, and the quantity on the right is nonnegative. This implies that both sides are $0$, meaning that we have full symmetry.
\end{proof}

\begin{rmk}
Note in particular that the statement and proof are in line with the general fact that (local) analytic $\Q$-factoriality implies $\Q$-factoriality, and is typically a stronger condition.
\end{rmk}

\subsection{Hodge-Du Bois numbers via threefold flips and flops}
According to Theorem 7.1 and Proposition 7.4 in \cite{PP25a}, a projective threefold $X$ with rational singularities exhibits the following 
identities between Hodge-Du Bois numbers and intersection Hodge numbers:
$$
\underline h^{p,q}(X)=I\underline h^{p,q}(X) \,\,\,\,\,\,{\rm for} \,\,\,\, (p,q)\neq (1,1)\mathrm{\;or\;}(1,2).
$$
For threefold terminal flops and flips, Koll\'ar \cite[Corollary 4.12]{Kollar89} proved that intersection cohomology is invariant. He also showed 
in \cite[Theorem 2.4]{Kollar89} that the local analytic type of the singularities is unchanged under a flop.

Here we show that, using these results, as a quick consequence of Theorem \ref{thm:analytic-Q-factoriality-threefolds} one can complete the picture for singular cohomology: all the Hodge-Du Bois numbers are invariant under a flop, while only  $\h^{1,2} (X)$ may change (in fact go up) under a flip. We split Theorem \ref{thm:invariance-Hodge} into the separate cases of flips and flops.

\subsubsection{Invariance of Hodge-Du Bois numbers under terminal flops}

\begin{cor}\label{cor:flops}
\label{cor: invariance of Hodge-Du Bois numbers under terminal flops}
Let $X$ be a projective threefold with $\Q$-factorial terminal singularities and $g:X\to Z$ be a flopping contraction. For a $\Q$-divisor $D\in \Div_\Q(X)$, let $X^+$ be the $D$-flop of $g$. If $X^+$ is $\Q$-factorial, then we have 
$$
\underline h^{p,q}(X)=\underline h^{p,q}(X^+) \,\,\,\,\,\,{\rm for ~all} \,\,\,\, 0\le p,q\le3.
$$
\end{cor}

Note that the $\Q$-factoriality of $X^+$ is immediate if $D$ is effective and $g$ is the contraction of a $(K_X+\epsilon D)$-negative extremal ray (see \cite[Proposition 3.37]{KM98}).

\begin{proof}
By Koll\'ar's work mentioned above, we have
$$
I\underline h^{p,q}(X)=I\underline h^{p,q}(X^+)
$$
for all $0\le p,q\le3$. Since $X$ and $X^+$ are $\Q$-factorial, we have
$$
\underline h^{1,1}(X)=\underline h^{2,2}(X)=  I\underline h^{2,2}(X) =  I\underline h^{1,1}(X)
=I\underline h^{1,1}(X^+) =  I\underline h^{2,2}(X^+) = \underline h^{2,2}(X^+) =\underline h^{1,1}(X^+),
$$
where the first and last equalities follow from Theorem \ref{thm:analytic-Q-factoriality-threefolds}. Therefore,
$$
\underline h^{p,q}(X)=\underline h^{p,q}(X^+)
$$
for $(p,q)\neq (1,2)$.  Now $X$ and $X^+$ have the same analytic singularity types, again by Koll\'ar's work mentioned above, and thus the sums of the defects of local analytic $\Q$-factoriality are the same:
$$
\sum_{s\in S}\sigma^{\rm an}(X;x_s)=\sum_{s^+\in S^+}\sigma^{\rm an}(X^+;x_{s^+}).
$$
In conclusion, we have
$$
\underline h^{1,2}(X)-\underline h^{2,1}(X)=\underline h^{1,2}(X^+)-\underline h^{2,1}(X^+)
$$
by Theorem \ref{thm:analytic-Q-factoriality-threefolds}, and therefore $\underline h^{1,2}(X)=\underline h^{1,2}(X^+)$.
\end{proof}

In particular, the Hodge-Du Bois numbers of two birational minimal models are the same, since by \cite[Theorem 5.3]{Kawamata88} or \cite[Theorem 4.9]{Kollar89}, any
birational map $f\colon X\dashrightarrow X'$ can be factored as a composition of $D$-flops with $\Q$-factorial terminal singularities.

\begin{cor}
Let $X$ and $X'$ be birational projective threefolds with $\Q$-factorial terminal singularities. If both $K_X$ and $K_{X'}$ are nef, then 
$$
\underline h^{p,q}(X)=\underline h^{p,q}(X')  \,\,\,\,\,\,{\rm for ~all} \,\,\,\, 0\le p,q\le3.
$$
\end{cor}

\subsubsection{Behavior of Hodge-Du Bois numbers under terminal flips}

\begin{cor}\label{cor:flip}
\label{thm: behavior of Hodge numbers under terminal flips}
Let $X$ be a projective threefold with $\Q$-factorial terminal singularities, and $g_R: X\to Z$ be the flipping contraction of a $K_X$-negative extremal ray. Let $g_R^+:X^+\to Z$ be the flip. Then,
$$
\underline h^{p,q}(X)=\underline h^{p,q}(X^+)
$$
for all $(p,q)\neq (1,2)$, and $\underline h^{1,2}(X)\le\underline h^{1,2}(X^+)$.
\end{cor}

\begin{proof}
Note that $X^+$ is $\Q$-factorial (see \cite[Proposition 3.37]{KM98}). Therefore, as in the proof of Corollary \ref{cor:flops}, we have
$$
\underline h^{p,q}(X)=\underline h^{p,q}(X^+)
$$
for $(p,q)\neq (1,2)$.

It remains to prove $\underline h^{1,2}(X)\le\underline h^{1,2}(X^+)$. Given Theorem \ref{thm:analytic-Q-factoriality-threefolds}, this is equivalent to 
the fact that the sum of the defects of local analytic $\Q$-factoriality does not decrease after a flip. This is essentially contained in \cite{KM92}, see for instance (11.9.1).

We include another argument for completeness, still following methods from \emph{loc. cit.}.
Let $C:={\rm Ex}(g_R)$ be the flipping curve, and $C^+:={\rm Ex}(g_R^+)$ the flipped curve. Then $X\sm C\isom X^+\sm C^+$. Recall that
$$
\chi(X)=\chi_c(X\sm C)+\chi(C), \quad \chi(X^+)=\chi_c(X^+\sm C^+)+\chi(C^+)
$$
where $\chi(\cdot )$ (resp. $\chi_c(\cdot)$) is the topological (resp. compactly supported topological) Euler characteristic. This implies that
$$
\chi(X)\ge \chi(X^+) \Longleftrightarrow \chi(C)\ge \chi(C^+).
$$

Recall by \cite[Section 8]{Kawamata88} or \cite[Theorem 1.2]{Mori88} that $Z$ has rational singularities. Denote by $C^{\rm sn}$ the seminormalization of $C$. We have the following commutative diagram
\begin{displaymath}
\xymatrix{
{R^1{g_R}_*\C_X}\ar[d]_{}\ar[r]^-{}& {H^1(C^{\rm sn}, \C)}\ar[d]_{}\\
{R^1{g_R}_*\O_X}\ar[r]^-{}&{H^1(C^{\rm sn},\O_{C^{\rm sn}})}
}
\end{displaymath}
where the first horizontal map is surjective by the proper base change theorem and the second vertical map is surjective by Hodge theory, namely $H^1(C^{\rm sn},\O_{C^{\rm sn}})=\gr^F_0H^1(C^{\rm sn}, \C)$. This implies that the second horizontal map is surjective, and thus $H^1(C^{\rm sn},\O_{C^{\rm sn}})=0$. Hence, $C^{\rm sn}$ consists of a tree of $\P^1$s. Likewise, the seminormalization of $C^+$ is a tree of $\P^1$s.
This gives
$$
\chi(C)=1+\#(\mathrm{irreducible\;components\;of\;} C)
$$
and likewise for $C^+$. By \cite[Theorem 13.5]{KM92}, we conclude that $\chi(C)\ge \chi(C^+)$, or equivalently $\chi(X)\ge \chi(X^+)$. Since all the other $\h^{p,q}$ are preserved, we obtain $\underline h^{1,2}(X)\le\underline h^{1,2}(X^+)$.
\end{proof}

We propose the following extension of Corollaries  \ref{cor:flops} and \ref{cor:flip} to higher dimension (we only focus on the non-obvious Hodge numbers).

\begin{conj}\label{conj:higher-dim-flops}
Let $X$ be a projective variety with $\Q$-factorial terminal singularities of dimension $n\ge 4$. In the setting of Corollary  \ref{cor:flops} we have 
$$
\underline h^{n-2,n-1}(X)= \underline h^{n-2,n-1}(X^+),
$$
while in the setting of Corollary  \ref{cor:flip} we have
$$
\underline h^{n-2,n-1}(X)\le \underline h^{n-2,n-1}(X^+).
$$
\end{conj}

We do not expect other Hodge-Du Bois numbers that are not close to the boundary to behave nicely under flips, since in general the  flipping and flipped loci may have different dimensions. 

%(for threefolds  they are both curves).

\section{Other applications}

\subsection{Variants of Samuel's conjecture}\label{scn:loc-an-Samuel}
We use Theorem \ref{thm:local analytic Q-factoriality defect} and the results in Section \ref{scn:loc-an-def} to deduce various versions of Samuel's conjecture in the context of (local analytic) $\Q$-factoriality. See \cite[Corollaire 3.14]{SGA2} for the proof of the original statement by Grothendieck, which says that a local complete intersection which is factorial in codimension $3$ is in fact factorial everywhere.

We start with the following more general version of Corollary \ref{cor:Samuel-lcdef-0} in the Introduction, which corresponds to the case 
$\lcdef (X) = 0$.

\begin{cor}\label{cor:Samuel-lcdef} 
Let $X$ be a normal algebraic variety, satisfying either $S_3$ or $\H^1(\DB_X^0)=0$. If $$\codim_X \,\Sing(X) \ge \lcdef (X) + 4,$$ then $X$ is locally analytically $\Q$-factorial.
\end{cor}

\begin{proof}
Recall that the RHM-defect object $\K^\bullet_X$ is supported on the singular locus $\Sing(X)$,  in cohomological degrees $\ge -\lcdef(X)$. Hence the codimension condition on the singular locus implies that the constructible cohomologies of the underlying $\Q$-complex satisfy
$$
\H^{\le-n+3}(\K^\bullet_X)=0,
$$
where $n=\dim X$. This implies that $\H^{-n+2}(\IC_X)=0$. Moreover, we have
$$
\H^{\le -n+3}\left(\gr^F_0\DR(\K^\bullet_X)\right)=0,
$$
which follows from the fact that the graded de Rham complex of a mixed Hodge module supported on a variety $Z$ is supported in cohomological degrees $\ge -\dim Z$. When $\H^1(\DB_X^0)=0$, the distinguished triangle
$$
\gr^F_0\DR(\K^\bullet_X)[-n]\to \DB_X^0\to I\DB_X^0\xrightarrow{+1}
$$
shows that $\H^{1}(I\DB_X^0)=R^1\mu_*\O_{\widetilde X}=0$ for a resolution $\mu\colon \widetilde X\to X$. When $X$ is $S_3$, Corollary \ref{cor:infinite-defect} implies that $\sigma^{\rm an}(X;x)$ is finite for every $x \in X$, since $X$ is smooth in codimension $3$. Therefore, Theorem \ref{thm:local analytic Q-factoriality defect} completes the proof.
\end{proof}

Most importantly, we prove a local analytic version of Samuel's conjecture. 

\begin{thm}
\label{thm:Samuel conjecture}
Let $X$ be a normal algebraic variety, which is $S_3$ and satisfies $\lcdef(X)=0$. Then  $X$ is locally analytically $\Q$-factorial if and only if 
it is locally analytically $\Q$-factorial away from a subvariety of codimension at least $4$.
\end{thm}

The statement applies in particular when $X$ is a local complete intersection, as in Corollary \ref{cor:Samuel conjecture} in the Introduction.

\begin{proof}
Since $\sigma^{\rm an}(X;x)= 0$ for $x$ away from a closed subset of codimension at least $4$, and $X$ is $S_3$, Corollary \ref{cor:infinite-defect} implies that $\sigma^{\rm an}(X;x)$ is finite for every $x \in X$, or equivalently $R^1\mu_*\O_{\widetilde X}=0$. 
Therefore the second part of Theorem \ref{thm:local analytic Q-factoriality defect} applies at every point of $X$.

We consider the distinguished triangle ($\ref{eqn:RHM-object}$): 
$$
\K_X^\bullet\to\Q^H_X[n]\to\IC^H_X\xrightarrow{+1}.
$$
Recall from \cite[Section 6]{PP25a} that the condition $\lcdef (X) = 0$ is equivalent to $\K_X^\bullet$ being supported only in degree $0$, while being a rational homology manifold at a point $x$ is equivalent to $(\K_X^\bullet)_{x} = 0$.

Now since $X$ has rational singularities in codimension $2$ (see the proof of Corollary \ref{cor:infinite-defect}), and it is well known that  rational surface singularities are rational homology manifolds, $\K^\bullet_X$ is supported on a closed subset $Y$ of codimension $\ge 3$. Then there exists a closed subset $Z\subset Y$ of codimension $\ge 4$ such that $Y\sm Z$ is smooth of pure dimension $n-3$ and 
$$\K^\bullet_X|_{Y\sm Z} \simeq \mathbb V[n-3],$$
where $\mathbb V$ is a local system  on $Y\sm Z$. By the hypothesis, we may further assume that $X\sm Z$ is locally analytically $\Q$-factorial. Denote by $j\colon Y\sm Z\hookrightarrow X$ the inclusion, and consider the distinguished triangle 
$$
\mathcal C^\bullet\to\K^\bullet_X\to j_*j^*\K^\bullet_X\xrightarrow{+1}.
$$
By the discussion above, the condition $\lcdef(X)=0$ implies that $\mathcal C^\bullet\in D^{\ge0}{\rm MHM}(X)$. Additionally, $\mathcal C^\bullet$ is supported on $Z$, and thus for the constructible cohomology in degree $-n+3$ we have $\H^{-n+3}(\mathcal C^\bullet)=0$. This implies the following inclusion of the constructible cohomologies of underlying $\Q$-complexes:
$$
\H^{-n+3}(\K^\bullet_X)=\H^{-n+2}(\IC_X)\hookrightarrow j_*\mathbb V.
$$

Let $x\in X$ be a closed point. By Theorem \ref{thm:local analytic Q-factoriality defect}, it suffices to prove that
$$
\H^{-n+2}(\IC_X)_x\to (R^2\mu_*\O_{\widetilde X})_x
$$
is injective. We argue by contradiction; suppose a nonzero element $\alpha\in\H^{-n+2}(\IC_X)_x$ maps to zero in $(R^2\mu_*\O_{\widetilde X})_x$. We lift $\alpha$ to a small neighborhood $U$ of $x$, so that $\alpha$ is a nonzero section of $\H^{-n+2}(\IC_U)$. Since $X\sm Z$ is locally analytically $\Q$-factorial, the map
$$
\H^{-n+2}(\IC_X)_y=\mathbb V_y\to (R^2\mu_*\O_{\widetilde X})_y
$$
is injective for all $y\in X\sm Z$. Since $\H^{-n+2}(\IC_X)\subset j_*\mathbb V$, there exist a point $y\in U\sm Z$ near $x$ such that $\alpha|_y\neq0$ and its image in $(R^2\mu_*\O_{\widetilde X})_y$ is zero. This is a contradiction.
\end{proof}

\subsection{Behavior in families}
Moving on to global statements, we start with two quick consequences of Theorem \ref{thm:Q-factoriality-defect}. The first establishes the constructibility of the defect of $\Q$-factoriality for projective families of varieties with rational singularities, which generalizes \cite[Proposition 12.1.7]{KM92} to arbitrary values of the $\Q$-factoriality defect in projective families.

\begin{cor}
\label{cor: constructibilty of Q-factoriality defect}
Let $f:\mathscr X\to S$ be a connected flat family of projective varieties with rational singularities over a variety $S$. For every nonnegative integer $m$, the set of closed points $s\in S$ with $\sigma(f^{-1}(s))=m$ is a constructible subset of $S$.
\end{cor}

\begin{proof}
Let $n$ be the fiber dimension of $f$. Since $R^2f_*\Q_{\mathscr X}$ and $R^{2n-2}f_*\Q_{\mathscr X}$ are constructible sheaves on $S$, the set of closed points $s\in S$ with
$$
\dim_\Q R^{2n-2}f_*\Q_{\mathscr X}|_s-\dim_\Q R^2f_*\Q_{\mathscr X}|_s=m
$$
is constructible. Therefore, we have the conclusion by Theorem \ref{thm:Q-factoriality-defect} and the proper base change theorem.
\end{proof}

The second consequence establishes that klt singularities of pairs deform when the ambient projective variety $X$ is fixed. Note that klt singularities do not deform in general (see e.g. \cite[Example 4.3]{Kawamata99}).

\begin{cor}
\label{cor: klt deforms X fixed}
Let $X$ be a normal projective variety and $f:(X\times S,D)\to S$ be a flat family of pairs $(X,D_s)$ over a variety $S$, where $D$ is an effective $\Q$-Weil divisor on $X\times S$. If $(X,D_{s_0})$ is klt for $s_0\in S$, then $(X,D_s)$ is klt for all $s\in S$ in a neighborhood of $s_0$.
\end{cor}

\begin{proof}
Since $X$ has rational singularities (see e.g. \cite[Theorem 5.22]{KM98}) and $K_X+D_{s_0}$ is $\Q$-Cartier, Theorem \ref{thm:Q-factoriality-defect} implies that $K_X+D_s$ is $\Q$-Cartier for all $s\in S$. Therefore, we conclude from Sato-Takagi \cite[Theorem A]{ST23}.
\end{proof}

\subsection{Bertini theorem for the defect of \texorpdfstring{$\Q$}{Q}-factoriality}
Here we record a ``bonus" consequence of Theorem \ref{thm:Q-factoriality-defect}; when combined with the Lefschetz hyperplane theorems in \cite{PP25a}, it can be used to deduce a restriction theorem for the defect of $\Q$-factoriality.

\begin{thm}
\label{thm: Bertini for defect of Q-factoriality}
Let $X$ be a normal projective variety of dimension $n\ge 4$. Let $L$ be an ample line bundle on $X$, with a base point free linear system $|V|\subset H^0(X,L)$. Then, for a general element $D\in |V|$, we have
$$
\sigma(D)\le\sigma(X).
$$
In particular, if $X$ is $\Q$-factorial, then $D$ is $\Q$-factorial.
Furthermore, if $\lcdef(X)\le n-4$ and $H^2(X, L^{-1})=0$, then
$$
\sigma(D)=\sigma(X).
$$
\end{thm}

Here $\lcdef (X)$ is the local cohomological defect of $X$; for a discussion and references, please see \cite[Section 2]{PP25a}.

Note however that the inequality $\sigma (D) \le \sigma (X)$ follows already from the work of Ravindra-Srinivas \cite{RS06}, who prove the stronger fact that, under our hypotheses, there is a Grothendieck-Lefschetz isomorphism of divisor class  groups ${\rm Cl} (X) \simeq {\rm Cl} (D)$. The second statement on the equality case 
is not very restrictive, as seen in the following: 

\begin{ex}\label{ex:sigma-eq}
Here are some examples of varieties for which the last condition holds: 

\noindent
$\bullet$~~First, by \cite[Theorem 1.3]{DT16} we know that the condition $\lcdef(X)\le n-4$ holds if ${\rm depth} (\O_X) \ge 4$ (so for instance any Cohen-Macaulay variety of dimension at least $4$) and the local analytic Picard groups of $X$ are torsion. 

\noindent
$\bullet$~~If $X$ is Du Bois with $\lcdef(X)\le n-4$, then the condition $H^2(X, L^{-1})=0$ holds automatically by \cite[Theorem 5.1]{PS24}. 

\noindent 
$\bullet$~~If $X$ is a complete intersection variety of dimension $\ge 4$ in a projective space $\P^N$, and $L=\O_X (1)$, then $\lcdef(X) = 0$, and the 
cohomology vanishing holds for straightforward reasons. In this context, it is worth recalling Grothendieck's theorem stating that any local complete intersection with $\codim_X\,\Sing(X) \ge 4$ is factorial; see also the beginning of Section \ref{scn:loc-an-Samuel}.
\end{ex}

\begin{proof}[Proof of Theorem \ref{thm: Bertini for defect of Q-factoriality}]
Recall  first  that $H^1(X,\Q)$ is pure of weight $1$, since  $X$ is normal. This implies that the natural map
$$
H^1(X,\Q)\to \mathit{IH}^1(X,\Q)
$$
is injective, since the kernel is of weight $\le 0$; see \cite[Remark 6.5]{PP25a}. In particular, $h^1(X)\le Ih^1(X)$.

By   Theorem \ref{thm:Q-factoriality-defect}, if $\sigma(X)$ is finite, then in fact $h^1(X)= Ih^1(X)$. Since $X$ is normal, hence ${\rm depth}(\O_X) \ge 2$, we have 
$\lcdef (X) \le n-2$ by a result of Ogus (see e.g. \cite[Example 2.5]{PP25a}), and therefore we have $h^1(X)\le h^1(D)$ by the singular weak Lefschetz statement \cite[Theorem B]{PP25a}.
Additionally, we have $Ih^1(X)= Ih^1(D)$ by the weak Lefschetz Theorem for intersection cohomology \cite[Section II.6.10]{GM88}. Thus, $Ih^1(D)\le h^1(D)$. 
Note however that $D$ is also a normal projective variety, and therefore $h^1(D) \le  Ih^1(D)$ by the argument in the previous paragraph. 
We conclude that $h^1(D) = Ih^1(D)$, hence $\sigma(D)$ is finite, again by  Theorem \ref{thm:Q-factoriality-defect}.

Therefore, it suffices to prove that $\sigma(D ) \le \sigma(X)$ when both are finite. Again by the weak Lefschetz for intersection cohomology, we have
$$
\mathit{IH}^2(X,\Q)\cap \mathit{IH}^{1,1}(X) \simeq \mathit{IH}^2(D,\Q)\cap \mathit{IH}^{1,1}(D).
$$
Applying Theorem \ref{thm:Q-factoriality-defect}, it remains to prove the following inequality
$$
\dim_\Q(\ker(H^2(X,\Q)\to H^2(X,\O_X)))\le \dim_\Q(\ker(H^2(D,\Q)\to H^2(D,\O_D))).
$$
To this end, consider the commutative diagram
\begin{displaymath}
\xymatrix{
{\Q_X^H[n]}\ar[d]_{}\ar[r]^-{}& {\IC_X^H}\ar[d]_{}\\
{\iota_*\Q_D^H[n]}\ar[r]^-{}&{\iota_*\IC_D^H[1]}
}
\end{displaymath}
where the vertical maps come from adjunction for $\iota:D\hookrightarrow X$; recall that there is an isomorphism $\iota^*\IC^H_X \simeq \IC^H_D[1]$ when $D$ is transverse to a Whitney stratification of $X$.
Taking hypercohomology, we obtain the commutative diagram:
\begin{displaymath}
\xymatrix{
{H^2(X,\Q)}\ar[d]_{}\ar[r]^-{}& {\mathit{IH}^2(X,\Q)}\ar[d]_{}\\
{H^2(D,\Q)}\ar[r]^-{}&{\mathit{IH}^2(D,\Q)}
}
\end{displaymath}
By \cite[Proposition 6.4]{PP25a}, the object $\K^\bullet_X$ in $(\ref{eqn:RHM-object})$ is of weight $\le n-1$, hence the kernels of the horizontal maps are of weight $\le 1$. This leads to the following commutative diagram:
\begin{displaymath}
\xymatrix{
{\gr^W_2H^2(X,\Q)}\ar[d]_{}\ar@{^{(}->}[r]& {\mathit{IH}^2(X,\Q)}\ar[d]_{}\\
{\gr^W_2H^2(D,\Q)}\ar@{^{(}->}[r]&{\mathit{IH}^2(D,\Q)}
}
\end{displaymath}
Since the right vertical map is an isomorphism as above, we deduce a natural inclusion
\begin{equation}
\label{eqn: injective restriction of weight 2 H2}
\gr^W_2H^2(X,\Q)\hookrightarrow \gr^W_2H^2(D,\Q).  
\end{equation}

Recall now from the proof of Theorem \ref{thm:Q-factoriality-defect} that the map $H^2(X,\Q)\to \gr^F_0H^2(X,\C)$ is the following composition
$$
H^2(X,\Q)\to H^2(X,\O_X)\to H^2(X,\DB^0_X).
$$
Therefore, applying Lemma \ref{lem: injectivity of (1,1) kernels restricted to weight 2}, we have a natural inclusion
$$
\ker(H^2(X,\Q)\to H^2(X,\O_X))\hookrightarrow\ker(\gr^W_2H^2(X,\Q)\to \gr^F_0(\gr^W_2H^2(X,\C))).
$$
This induces the commutative diagram
\begin{displaymath}
\xymatrix{
{\ker(H^2(X,\Q)\to H^2(X,\O_X))}\ar[d]_{}\ar@{^{(}->}[r]& {\ker(\gr^W_2H^2(X,\Q)\to \gr^F_0(\gr^W_2H^2(X,\C)))}\ar[d]_{}\\
{\ker(H^2(D,\Q)\to H^2(D,\O_D))}\ar@{^{(}->}[r]&{\ker(\gr^W_2H^2(D,\Q)\to \gr^F_0(\gr^W_2H^2(D,\C)))}
}
\end{displaymath}
where the horizontal maps are injective and the right vertical map is injective by \eqref{eqn: injective restriction of weight 2 H2}. Consequently, the left vertical map is injective, from which we finally conclude that
$$
\sigma(D)\le\sigma(X).
$$

Assume now that $\lcdef(X)\le n-4$ and $H^2(X, L^{-1})=0$. Using again \cite[Theorem B]{PP25a} we have $h^1(X)=h^1(D)$, while the 
weak Lefschetz for intersection cohomology gives $Ih^1(X)=Ih^1(D)$. Additionally, in the following diagram,
\begin{displaymath}
\xymatrix{
{H^2(X,\Q)}\ar[d]_{}\ar[r]^-{}& {H^2(X,\O_X)}\ar[d]_{}\\
{H^2(D,\Q)}\ar[r]^-{}&{H^2(D,\O_D)}
}
\end{displaymath}
the left vertical map is an isomorphism, also by \cite[Theorem B]{PP25a}, while the right vertical map is injective by the assumption 
$H^2(X, L^{-1})=0$. This gives
$$
\ker(H^2(X,\Q)\to H^2(X,\O_X))=\ker(H^2(D,\Q)\to H^2(D,\O_D)),
$$
and therefore Theorem \ref{thm:Q-factoriality-defect} implies that
$\sigma(D)=\sigma(X)$.
\end{proof}

\subsection{Factoriality of hypersurfaces in fourfolds}\label{scn:fourfolds}
 The paper \cite{PRS14} gives a criterion for the factoriality of a hypersurface $X$ in a smooth complete intersection fourfold $Y$ in $\P^n$, when $X$ has only isolated ordinary multiple points. We show that Theorem \ref{thm:Q-factoriality-defect} leads to a quick proof of the generalization of this statement to arbitrary smooth projective fourfolds $Y$, by first extending the main technical statement of \emph{loc. cit.}, namely \cite[Theorem A]{PRS14},  to this setting.

\begin{thm}
\label{thm: Q-factoraility of ample divisor}
Let $X\subset Y$ be an irreducible normal ample divisor in a smooth projective fourfold $Y$. Let $\widetilde Y$ be the blow-up of $Y$ along finitely many points on $X$, 
and let  $\widetilde X\subset \widetilde Y$ be the strict transform of $X$. If $\widetilde X$ is a normal $\Q$-factorial ample divisor in $\widetilde Y$, then $X$ is $\Q$-factorial.
\end{thm}

\begin{proof}
By the weak Lefschetz theorem, we have an isomorphism given by the restriction map
$$
H^2(Y,\Q)\isom H^2(X,\Q),
$$
and similarly
$$
H^2(Y,\O_Y)\isom H^2(X,\O_X),
$$
using the Kodaira vanishing theorem. Therefore, we obtain an equality
$$
\dim_\Q (\ker(H^2(Y,\Q)\to H^2(Y,\O_Y)))=\dim_\Q (\ker(H^2(X,\Q)\to H^2(X,\O_X))).
$$

Assuming that $\widetilde X$ is a normal $\Q$-factorial ample divisor in $\widetilde Y$, we first prove that $\sigma(X)$ is finite. Since the first cohomology is invariant under smooth blow-ups, we have
$$
H^1(X,\Q)\isom H^1(Y,\Q)\isom H^1(\widetilde Y,\Q) \isom H^1(\widetilde X,\Q),
$$
where the first and third isomorphisms are given by the weak Lefschetz theorem. Therefore, using Lemma \ref{lem: H1 and H2 of normal variety}, we have
$$
h^1(X)=h^1(\widetilde X)=Ih^1(\widetilde X)=Ih^1(X)=h^5(X).
$$
Hence, $\sigma(X)$ is finite by Theorem \ref{thm:Q-factoriality-defect}.

We next show that $\sigma(X)=0$. As above, we have an equality
$$
\dim_\Q (\ker(H^2(\widetilde Y,\Q)\to H^2(\widetilde Y,\O_{\widetilde Y})))=\dim_\Q (\ker(H^2(\widetilde X,\Q)\to H^2(\widetilde X,\O_{\widetilde X}))).
$$
Since both $\widetilde Y$ and $\widetilde X$ are $\Q$-factorial, they are equal to
$$
\dim_\Q(\mathit{IH}^2(\widetilde Y,\Q)\cap \mathit{IH}^{1,1}(\widetilde Y))=\dim_\Q(\mathit{IH}^2(\widetilde X,\Q)\cap \mathit{IH}^{1,1}(\widetilde X))
$$
by Theorem \ref{thm:Q-factoriality-defect}. Let $I$ be the 
index set for the exceptional divisors of $\widetilde Y\to Y$. Using Proposition \ref{prop: kernel of H2 to IH2},  it is easy to see  that 
$$
\dim_\Q(\mathit{IH}^2(Y,\Q)\cap \mathit{IH}^{1,1}(Y))=\dim_\Q(\mathit{IH}^2(\widetilde Y,\Q)\cap \mathit{IH}^{1,1}(\widetilde Y))-|I|
$$
and
$$
\dim_\Q(\mathit{IH}^2(X,\Q)\cap \mathit{IH}^{1,1}(X))\le\dim_\Q(\mathit{IH}^2(\widetilde X,\Q)\cap \mathit{IH}^{1,1}(\widetilde X))-|I|.
$$
The latter inequality follows from the fact that the morphism $\widetilde X\to X$ has at least $|I|$ exceptional divisors. By combining the above equalities and inequalities, we obtain
$$
\dim_\Q(\mathit{IH}^2(X,\Q)\cap \mathit{IH}^{1,1}(X))\le \dim_\Q (\ker(H^2(X,\Q)\to H^2(X,\O_X))),
$$
from which we conclude $\sigma(X)\le 0$ by Theorem \ref{thm:Q-factoriality-defect}. Hence, $X$ is $\Q$-factorial.
\end{proof}

The theorem is particularly useful when $X$ is a section of a power of a very ample line bundle with ordinary multiple point singularities; this leads to the promised extension of \cite[Theorem B]{PRS14} to arbitrary smooth fourfolds.

\begin{cor}
\label{cor: factoriality with ordinary multiple points}
Let $L$ be a very ample line bundle on a smooth projective fourfold $Y$, and $X$ is an irreducible normal divisor of $|dL|$. Let $S=\left\{p_1,...,p_k\right\}$ be the set of isolated ordinary multiple points of $X$ of multiplicities $m_1,...,m_k$. If $X\sm S$ is factorial and
$$
\sum_{i=1}^km_i<d,
$$
then X is factorial.
\end{cor}

Once we have Theorem \ref{thm: Q-factoraility of ample divisor}, the nice argument given in the proof of \cite[Theorem 4.1]{PRS14} works identically: the hypotheses ensure the ampleness of the proper transform 
$\widetilde X$ on the blow-up $\widetilde Y$ of $Y$ along $S$. (Note that it suffices to assume that $X\sm S$ is factorial, since the fact that $S$ consists of ordinary singularities implies that $\widetilde X$ is smooth 
in the neighborhood of $S$.)

\iffalse
\begin{proof}
Since the local Picard group of a three dimensional isolated hypersurface singularity is torsion free (see Robbiano \cite{Robbiano76}), it suffices to prove that $X$ is $\Q$-factorial. Let $\widetilde Y$ be the blow-up of $Y$ along $S$, and let  $\widetilde X\subset \widetilde Y$ the strict transform of $X$.

By Theorem \ref{thm: Q-factoraility of ample divisor}, it suffices to prove that $\widetilde X$ is a normal $\Q$-factorial ample divisor on $\widetilde Y$. Since $S$ consists of ordinary multiple points, $\widetilde X$ is smooth in a neighborhood of $S$. Using the fact that $X\sm S$ is factorial, $\widetilde X$ is a normal factorial divisor.

The ampleness of the divisor $\widetilde X$ in $\widetilde Y$ follows exactly by the argument given in the proof of \cite[Theorem 4.1]{PRS14}.
\end{proof}

\fi

\begin{rmk}
In particular, if $S$ is a set consisting of at most $\lfloor \frac{d-1}{2}\rfloor$ ordinary double points (i.e. nodes), then $X$ is factorial. For degree $d$ hypersurfaces in $\P^4$ with only nodes, Cheltsov \cite{Cheltsov10} proved that $X$ is factorial when there are at most $(d-1)^2-1$ nodes, and provided an example of a non-factorial variety with $(d-1)^2$ nodes. It would be interesting to understand sharp bounds on the number of singular points required to guarantee factoriality in general, in the context of Corollary \ref{cor: factoriality with ordinary multiple points}; cf. \cite[Conjecture D]{PRS14} when $Y = \P^4$.
\end{rmk}

\bibliographystyle{alpha}
\bibliography{refs}

\begin{thebibliography}{HTT08}

\bibitem[Bat98]{Batyrev98}
Victor~V. Batyrev.
\newblock Stringy {H}odge numbers of varieties with {G}orenstein canonical singularities.
\newblock In {\em Integrable systems and algebraic geometry ({K}obe/{K}yoto, 1997)}, pages 1--32. World Sci. Publ., River Edge, NJ, 1998.

\bibitem[BO74]{BO74}
Spencer Bloch and Arthur Ogus.
\newblock Gersten's conjecture and the homology of schemes.
\newblock {\em Ann. Sci. \'{E}cole Norm. Sup. (4)}, 7:181--201 (1975), 1974.

\bibitem[BS00]{BS00}
J.~Biswas and V.~Srinivas.
\newblock A {L}efschetz {$(1,1)$} theorem for normal projective complex varieties.
\newblock {\em Duke Math. J.}, 101(3):427--458, 2000.

\bibitem[Che10]{Cheltsov10}
Ivan Cheltsov.
\newblock Factorial threefold hypersurfaces.
\newblock {\em J. Algebraic Geom.}, 19(4):781--791, 2010.

\bibitem[DT16]{DT16}
Hailong Dao and Shunsuke Takagi.
\newblock On the relationship between depth and cohomological dimension.
\newblock {\em Compos. Math.}, 152(4):876--888, 2016.

\bibitem[FL24a]{FL22}
Robert Friedman and Radu Laza.
\newblock Deformations of some local {C}alabi-{Y}au manifolds.
\newblock {\em \'Epijournal de G\'eom\'etrie Alg\'ebrique, Special volume in honour of Claire Voisin (November 5, 2024) epiga:10899}, 03 2024.

\bibitem[FL24b]{FL24a}
Robert Friedman and Radu Laza.
\newblock Higher {D}u {B}ois and higher rational singularities.
\newblock {\em Duke Math. J.}, 173(10):1839--1881, 2024.
\newblock Appendix by Morihiko Saito.

\bibitem[Fle81]{Flenner81}
Hubert Flenner.
\newblock Divisorenklassengruppen quasihomogener {S}ingularit\"aten.
\newblock {\em J. Reine Angew. Math.}, 328:128--160, 1981.

\bibitem[Ful98]{Fulton98}
William Fulton.
\newblock {\em Intersection theory}, volume~2 of {\em Ergebnisse der Mathematik und ihrer Grenzgebiete. 3. Folge. A Series of Modern Surveys in Mathematics [Results in Mathematics and Related Areas. 3rd Series. A Series of Modern Surveys in Mathematics]}.
\newblock Springer-Verlag, Berlin, second edition, 1998.

\bibitem[GM88]{GM88}
Mark Goresky and Robert MacPherson.
\newblock {\em Stratified {M}orse theory}, volume~14 of {\em Ergebnisse der Mathematik und ihrer Grenzgebiete (3) [Results in Mathematics and Related Areas (3)]}.
\newblock Springer-Verlag, Berlin, 1988.

\bibitem[GR58]{GR58}
Hans Grauert and Reinhold Remmert.
\newblock Bilder und {U}rbilder analytischer {G}arben.
\newblock {\em Ann. of Math. (2)}, 68:393--443, 1958.

\bibitem[Gra62]{Grauert62}
Hans Grauert.
\newblock \"uber {M}odifikationen und exzeptionelle analytische {M}engen.
\newblock {\em Math. Ann.}, 146:331--368, 1962.

\bibitem[Gro68]{SGA2}
Alexander Grothendieck.
\newblock {\em Cohomologie locale des faisceaux coh\'erents et th\'eor\`emes de {L}efschetz locaux et globaux {$(SGA$} {$2)$}}, volume Vol. 2 of {\em Advanced Studies in Pure Mathematics}.
\newblock North-Holland Publishing Co., Amsterdam; Masson \& Cie, Editeur, Paris, 1968.
\newblock Augment\'e{} d'un expos\'e{} par Mich\`ele Raynaud, S\'eminaire de G\'eom\'etrie Alg\'ebrique du Bois-Marie, 1962.

\bibitem[GW18]{GW2018}
Antonella Grassi and Timo Weigand.
\newblock On topological invariants of algebraic threefolds with ({Q}-factorial) singularities.
\newblock {\em preprint arXiv:1804.02424}, 2018.

\bibitem[HTT08]{HTT}
Ryoshi Hotta, Kiyoshi Takeuchi, and Toshiyuki Tanisaki.
\newblock {\em {$D$}-modules, perverse sheaves, and representation theory}, volume 236 of {\em Progress in Mathematics}.
\newblock Birkh\"{a}user Boston, Inc., Boston, MA, {J}apanese edition, 2008.

\bibitem[Jan90]{Jannsen90}
Uwe Jannsen.
\newblock {\em Mixed motives and algebraic {$K$}-theory}, volume 1400 of {\em Lecture Notes in Mathematics}.
\newblock Springer-Verlag, Berlin, 1990.
\newblock With appendices by S. Bloch and C. Schoen.

\bibitem[Kaw88]{Kawamata88}
Yujiro Kawamata.
\newblock Crepant blowing-up of {$3$}-dimensional canonical singularities and its application to degenerations of surfaces.
\newblock {\em Ann. of Math. (2)}, 127(1):93--163, 1988.

\bibitem[Kaw99]{Kawamata99}
Yujiro Kawamata.
\newblock On the extension problem of pluricanonical forms.
\newblock In {\em Algebraic geometry: {H}irzebruch 70 ({W}arsaw, 1998)}, volume 241 of {\em Contemp. Math.}, pages 193--207. Amer. Math. Soc., Providence, RI, 1999.

\bibitem[KM92]{KM92}
J\'anos Koll\'ar and Shigefumi Mori.
\newblock Classification of three-dimensional flips.
\newblock {\em J. Amer. Math. Soc.}, 5(3):533--703, 1992.

\bibitem[KM98]{KM98}
J\'anos Koll\'ar and Shigefumi Mori.
\newblock {\em Birational geometry of algebraic varieties}, volume 134 of {\em Cambridge Tracts in Mathematics}.
\newblock Cambridge University Press, Cambridge, 1998.
\newblock With the collaboration of C. H. Clemens and A. Corti, Translated from the 1998 Japanese original.

\bibitem[Kol]{Kollar}
J\'{a}nos Koll\'{a}r.
\newblock Unpublished draft on local {P}icard groups.

\bibitem[Kol89]{Kollar89}
J\'anos Koll\'ar.
\newblock Flops.
\newblock {\em Nagoya Math. J.}, 113:15--36, 1989.

\bibitem[Mor88]{Mori88}
Shigefumi Mori.
\newblock Flip theorem and the existence of minimal models for {$3$}-folds.
\newblock {\em J. Amer. Math. Soc.}, 1(1):117--253, 1988.

\bibitem[NS95]{NS95}
Yoshinori Namikawa and J.~H.~M. Steenbrink.
\newblock Global smoothing of {C}alabi-{Y}au threefolds.
\newblock {\em Invent. Math.}, 122(2):403--419, 1995.

\bibitem[Par23]{Park23}
Sung~Gi Park.
\newblock {Du Bois complex and extension of forms beyond rational singularities}.
\newblock {\em preprint arXiv:2311.15159}, 2023.

\bibitem[PP24]{PP}
Sung~Gi Park and Mihnea Popa.
\newblock Lefschetz theorems, {Q}-factoriality, and {H}odge symmetry for singular varieties.
\newblock {\em preprint arXiv:2410.15638}, 2024.

\bibitem[PP25]{PP25a}
Sung~Gi Park and Mihnea Popa.
\newblock {H}odge symmetry and {L}efschetz theorems for singular varieties.
\newblock {\em preprint}, 2025.

\bibitem[PRS14]{PRS14}
Francesco Polizzi, Antonio Rapagnetta, and Pietro Sabatino.
\newblock On factoriality of threefolds with isolated singularities.
\newblock {\em Michigan Math. J.}, 63(4):781--801, 2014.

\bibitem[PS08]{PS08}
Chris A.~M. Peters and Joseph H.~M. Steenbrink.
\newblock {\em Mixed {H}odge structures}, volume~52 of {\em Ergebnisse der Mathematik und ihrer Grenzgebiete. 3. Folge. A Series of Modern Surveys in Mathematics [Results in Mathematics and Related Areas. 3rd Series. A Series of Modern Surveys in Mathematics]}.
\newblock Springer-Verlag, Berlin, 2008.

\bibitem[PS25]{PS24}
Mihnea Popa and Wanchun Shen.
\newblock Du {B}ois complexes of cones over singular varieties, local cohomological dimension, and {K}-groups.
\newblock {\em Rev. Roumaine Math. Pures Appl., L. B\u adescu memorial issue}, 70(1-2):133--155, 2025.

\bibitem[RS06]{RS06}
G.~V. Ravindra and V.~Srinivas.
\newblock The {G}rothendieck-{L}efschetz theorem for normal projective varieties.
\newblock {\em J. Algebraic Geom.}, 15(3):563--590, 2006.

\bibitem[Sai90]{Saito90}
Morihiko Saito.
\newblock Mixed {H}odge modules.
\newblock {\em Publ. Res. Inst. Math. Sci.}, 26(2):221--333, 1990.

\bibitem[Sai00]{Saito00}
Morihiko Saito.
\newblock Mixed {H}odge complexes on algebraic varieties.
\newblock {\em Math. Ann.}, 316(2):283--331, 2000.

\bibitem[ST23]{ST23}
Kenta Sato and Shunsuke Takagi.
\newblock Deformations of log terminal and semi log canonical singularities.
\newblock {\em Forum Math. Sigma}, 11:Paper No. e35, 34, 2023.

\end{thebibliography}

\end{document}